\documentclass[11pt,leqno]{amsart}
\usepackage[a4paper,asymmetric]{geometry}

\usepackage{latexsym,amssymb}
\usepackage{hyperref}
\usepackage{amsmath,amsthm}
\usepackage{amsbsy}
\usepackage{amsfonts}
\usepackage[american]{babel}
\usepackage{mathrsfs,enumerate}
\usepackage{psfrag}
\usepackage{graphicx}
\usepackage{xcolor}
\usepackage{fullpage}
\usepackage{tikz}
\usetikzlibrary{patterns}

\newcommand{\isopf}{isoperimetric function }
\newcommand{\isopi}{isoperimetric inequality }
\newcommand{\isopp}{isoperimetric problem }
\newcommand{\logi}{logistic }

\newcommand{\poinc}{Poincar\'e }
\newcommand{\N}{\mathbb{N}}
\newcommand{\R}{\mathbb{R}}
\newcommand{\rn}{{\mathbb{R}^N}}
\newcommand{\bd}{\partial}
\newcommand{\bm}[1]{{#1}^+}

\newcommand{\C}[1]{\textbf{C}_{#1}}
\newcommand{\cc}{{\sf{C}}}

\newcommand{\ee}{\text{e}}
\newcommand{\ep}{\varepsilon}

\newcommand{\fii}{\varphi}
\newcommand{\ff}{\,\textbf{f}\,}
\newcommand{\fg}{\,\textbf{h}\,}

\newcommand{\hh}{\mathscr{H}}

\newcommand{\I}[1]{{\rm{\bf{I}}}_{#1}}

\newcommand{\la}{\lambda}
\newcommand{\lab}{\labb_k}
\newcommand{\labb}{\overline{\boldsymbol\la}}
\newcommand{\lo}{{\boldsymbol\mu}}
\newcommand{\lon}{\lo^N}

\newcommand{\mun}{\mu^N}

\newcommand{\nne}{\nu_{\ep}}

\newcommand{\pp}{\alpha}%%{\vartheta}

\newcommand{\qq}{\text{Q}}
\newcommand{\qn}[2]{\textsf{Q}^{#1}(#2)}

\newcommand{\sn}[1]{\mathbb{S}^{#1}}

\newcommand{\te}{\tau_{\ep}}

\newcommand{\vv}{\textsf{v}}
\newcommand{\va}{\textsf{Var}}
\newcommand{\var}[1]{{\rm{\sf{Var}}_{#1}}}
\newcommand{\vn}[1]{\textsf{V}_{#1}}

\newcommand{\wud}[1]{\text{W}^{1,2}_{#1}}

\newcommand{\wwb}{\mathscr{W}^b}

\newcommand{\zn}{Z_N}

\numberwithin{equation}{section}
\newtheorem{theorem}{Theorem}[section]
\newtheorem{proposition}[theorem]{Proposition}
\newtheorem{corollary}[theorem]{Corollary}

\newtheorem{remark}[theorem]{Remark}

\begin{document}
%%%%%%%%%%%%%%%%%%%%%%%%%%%%%%%%%%%%%%%%%%%%%%%%%%%%%%%%%%%%%%%%%%%%%%%%%%%%%%%%%%%%%%%%%%%%%%%%%%%%%%%%%%%%%%%%%%%%%%%%%%%%%%%%%%%%%%%%
\title[]{Isoperimetry and stability of hyperplanes \\ for product probability measures}

\author[]{F. Barthe}
\author[]{C. Bianchini}
\author[]{A. Colesanti}

\address{{\sf{Franck Barthe}}: {Institut de Math\'ematiques de Toulouse,  Universit\'e Paul Sabatier, 31062 Toulouse cedex 9, France}}
\email{franck.barthe@math.univ-toulouse.fr}

\address{{\sf{Chiara Bianchini}}: Institut Elie Cartan, Universit\'e Henri Poincar\'e Nancy, Boulevard des Aiguillettes B.P. 70239, F-54506 Vandoeuvre-les-Nancy Cedex, France}
\email{cbianchini@math.unifi.it}

\address{{\sf{Andrea Colesanti}}: Dip.to di Matematica ``U. Dini'', Universit\`a degli Studi di Firenze, Viale Morgagni 67/A, 50134 Firenze, Italy}
\email{andrea.colesanti@math.unifi.it}

\address{}
\email{}
\date{}


%%%%%%%%%%%%%%%%%%%%%%%%%%%%%%%%%%%%%%%%%%%%%%%%%%%%%%%%%%%%%%%%%%%%%%%%%%%%%%%%%%%%%%%%%%%%%%%%%%%%%%%%%%%%%%%%%%%%%%%%%%%%%%%%%%%%%%%%

%%%%%%%%%%%%%%%%%
\begin{abstract}
%%%%%%%%%%%%%%%%
We investigate stationarity and stability of half-spaces as isoperimetric sets
for product probability measures, considering the cases of coordinate and non-coordinate half-spaces.
Moreover, we present several examples to which our results can be applied, with a particular emphasis on the logistic measure.
\end{abstract}

\maketitle

%%%%%%%%%%%%%%%%%%%%%%%%
\section{Introduction}
%%%%%%%%%%%%%%%%%%%%%%%%
The isoperimetric problem consists in finding sets of given volume and minimal boundary measure.
It has been intensively studied in Euclidean and Riemannian geometry and similar questions have been also investigated for probability measures, in relation with deviation and concentration inequalities. 
In particular there have been many interactions between the geometric and probabilistic viewpoints, see e.g. \cite{Ba2,BH,L1,M,R}.

The isoperimetric problem may be formulated in a metric measure space $(X,d,\tau)$ considering the so called isoperimetric function. 
Given a Borel subset $A\subset X$ and $h>0$, the $h$-enlargement of $A$ is  $A_h:=\{ x\in X:\, d(x,A)\le h \}$. The  \emph{boundary measure} of $A$ can be defined
as the Minkowski content
$$
\bm{\tau}(\bd A)= \lim_{h\to 0^+}\frac{\tau(A_h\setminus A)}{h}.
$$
We call {\em\isopf} of $\tau$ the function
$$
\I{\tau}(y) = \inf \{ \bm{\tau}(\bd A)\ |\ \tau(A)=y \},
$$
defined on $[0,\tau(X)]$.
We say that  $A\subset  X$  is an \emph{isoperimetric set}, or solves the isoperimetric problem for $\tau$,  if $\I{\tau}\big(\tau(A)\big)=\bm{\tau}(\bd A)$, meaning that $A$ minimizes the boundary measure among all the subsets with the same measure.

Determining isoperimetric sets is a beautiful problem, but often too difficult.
A natural approach is to look first for sets which look like minimizers of the boundary measure,
for infinitesimal volume preserving deformations.
More precisely, one can consider the so-called stationarity and stability properties of an open set, which correspond to the fact that, under the action of measure preserving  perturbations, the first variation of the boundary measure vanishes and the second variation is non-negative, respectively (see \cite{Ba, G}).
In \cite{RCBM}, it is proved that these conditions have rather simple explicit analytic characterization, which will be described in details in the next sections.

\medskip

In this note we consider product probability measures on Euclidean spaces (which amount to independent random variables). 
To be more precise, we consider a probability measure $\tau$ on $\R^d$ with a density with respect to the $d$-dimensional Lebesgue measure: $d\tau(x)=f(x)\, dx$. 
Through most of paper we will consider only positive and $C^2$ density functions $f$ (and often we write them as $\ee^{-V}$  or $\ee^{\psi}$ with $V, \psi\in C^2$). For $N\in\N$ we denote by $\tau^N$ the \emph{$N$-fold product measure} of $\tau$ on $\R^{dN}$: $d\tau^N(x_1,...,x_N)=d\tau(x_1)\cdot...\cdot d\tau(x_N)$, with
 $x_1,...,x_N\in\R^d$.

Using the characterization proved in \cite{RCBM}, we investigate stationarity and stability of half-spaces for $\mu^{N+1}$. 
We first deal with stationarity, on which the main result is Corollary \ref{stazcoroll}, that provides a characterization of stationary half-spaces for $\tau^N$.
In particular this result shows that, coordinate half-spaces apart and with the exception of the Gaussian measure, there are very few possible stationary half-spaces. 
Subsequently, in Theorem \ref{stabteocoordinate} and Theorem \ref{stabteononcoordinate}, we establish some conditions allowing to select stable half-spaces among those which are stationary.

More precisely Theorem \ref{stabteocoordinate} concerns coordinate half-spaces (and the simple case of Gauss measure); in this case the stability condition involves the so-called spectral gap of the measure $\mu$ (i.e. the best constant in the Poincar\'e inequality) and no special assumptions are requested on the measure. 
On the contrary Theorem \ref{stabteononcoordinate} is about non-coordinate (stationary) half-spaces for log-concave measure. 
Roughly speaking we prove that in the non-coordinate case, stability in any dimension is equivalent to stability (of projections) in dimension three, and stability in dimension three implies stability (of projections) in dimension two. 
This last implication can not be reversed in general, as it is pointed out in the study of the logistic measure, made in Section \ref{sec:ex}. 

The notion of spectral gap  plays a crucial role for the stability issues. 
In particular Section \ref{poincaresec} is devoted to prove a tensorization result for weighted Poincar\'e inequalities, which represents a crucial step in the proof of Theorem \ref{stabteononcoordinate}.

\medskip

The purpose of the second part of this paper, i.e. Section \ref{sec:ex}, is to illustrate by some examples
the variety of situations that may occur, according to the choice of the measure $\mu$, concerning stable half-spaces. In particular, we construct measures for which the only stable half-spaces are coordinate, and we show that for suitable perturbations of the Gaussian measure stable non-coordinate hyperplanes exist in any dimensions.

A special place of this part is taken up by the study of the logistic measure $\lo$ on $\R$, whose density is % $d\lo(x)=f(x)\, dx$ where
$$
f(x)=\frac{\ee^x}{(1+\ee^x)^2}, \quad x\in \R.
$$
Our first result is to find the exact value of the spectral gap of $\lo$ (see Proposition \ref{calcolola}):
$$
\lambda_\lo=\frac 14,
$$
which allows us to estimate the \isopf of the product measure in terms of the \isopf of the generating one-dimensional measure. 
This problem is worth to be studied for a general probability measure $\mu$. 
Indeed one can easily prove that for every $N\ge 1$ and for every $0\le t\le 1$,
\begin{equation}\label{ison}
\I{\mu^{N+1}}(t) \le \I{\mu^N}(t)\le \dots \le \I{\mu}(t).      
\end{equation}
A less trivial task is to find a corresponding lower bound, i.e. a constant $\C{\mu}$ such that for every $N\ge 1$ and for every $0\le t \le 1$:
\begin{equation}\label{isoGeneric}
\I{\mu^N}(t) \ge \C{\mu} \ \I{\mu}(t).
\end{equation}
Roughly speaking, this inequality means that if a set $A\subseteq\R^{d}$ minimizes the boundary measure between all the subsets of $\R^{d}$ with fixed measure $\mu^d(A)$, then the $(d+m)$-dimensional cylinder associated to $A$ solves again the isoperimetric problem for $\I{\mu^{d+m}}$, up to a factor $\C{\mu}$. 
More generally such a dimension-free isoperimetric inequality leads to an estimate of the so called infinite dimensional isoperimetric function $\I{\mu^{\infty}}$ defined as
$$
\I{\mu^\infty}(t)=\inf_{N\in\N}\I{\mu^N}(t).
$$
A refinement of this estimate can be done comparing $\I{\mu^{\infty}}$ with the \isopf of the Gaussian measure (see Proposition \ref{muinfinity}).

Sufficient conditions on the measure $\mu$ which guarantee the validity of (\ref{isoGeneric}) are known (see \cite{BCR1},\cite{BCR2}, \cite{eM}).
The most famous example is given by the Gaussian measure
$$
d\gamma(x)=\frac 1{\sqrt{2\pi\sigma^2}}\ee^{-\frac{|x|^2}{2\sigma}},
$$
which satisfies inequality (\ref{isoGeneric}) with $\C{\gamma}=1$, that is: $\I{\gamma^N}(t)=\I{\gamma}(t)$ for every $0\le t\le 1$ (\cite{Bo}). In this case half-spaces are isoperimetric sets. Another important example is due to Bobkov and Houdr\'e, who proved in \cite{BH}, that the exponential measure
$$
d\nu(x)=\frac 12 \ee^{-|x|},
$$
satisfies (\ref{isoGeneric}) with $\C{\nu}=\frac 1{2\sqrt{6}}$. In \cite{BCR2} other examples are given; in particular it is proved that if $V$ is a symmetric non-negative convex function with $\sqrt{V}$ concave in the large, then such an inequality (\ref{isoGeneric}) holds. 

In Theorem \ref{teoisoplogi} we establish the validity of (\ref{isoGeneric}) for the logistic measure, with an explicit value of the constant $\C{\lo}$. 
%%This leads in turn to a lower bound for the so-called infinite dimensional isoperimetric function of $\lo$. 
Finally, applying the results of the first part of the paper, we give a detailed description of stationary and stable half-spaces for the logistic measure.

%%%%%%%%%%%%%%%%%%%%%%%%%%%%%%%%%%%%%%%%%%%%%%%%%%%%%%%%%%%%%%%%%%%%%%%%%%%%%%%%%%%%%%%%%%%%%%%%%%%%%%%%%%%%%%%%%%%%%%%%%
\section{Spectral gap and Poincar\'e type inequalities}\label{poincaresec}
%%%%%%%%%%%%%%%%%%%%%%%%%%%%%%%%%%%%%%%%%%%%%%%%%%%%%%%%%%%%%%%%%%%%%%%%%%%%%%%%%%%%%%%%%%%%%%%%%%%%%%%%%%%%%%%%%%%%%%%%%
Let $\tau$ be a probability measure $\tau$ $\R^d$. For $u\in\text{L}^{2}_{\tau}(\R^d)$ we denote by $\var{\tau}(u)$ its variance with respect to $\tau$. 
The \emph{spectral gap} $\lambda_\tau$ of $\tau$, also called its \emph{Poincar\'e constant}, is by definition the largest  constant $\la$ such that for every $u\in\text{W}^{1,2}_{\tau}(\R^d)$
\begin{equation}\label{poinc}
\la\ \var{\tau}(u)\le \int_{\R^d}|Du|^2\,d\tau(x).
\end{equation}
Clearly $\la_\tau\ge0$.

%%%%%%%%%%%-----------------------------------------------------------------------------------------------------------------------------
\subsection{Log-concave probability measures}
%%%%%%%%%%%-----------------------------------------------------------------------------------------------------------------------------
It was proved in \cite{B2} that log-concave probability measures $\tau$ satisfy a non-trivial Poincar\'e inequality, that is $\lambda_\tau>0$. The precise estimation of this positive value is still a topic of investigation. In this 
paper these quantitative bounds will not be needed.

However, we will often use another Poincar\'e type inequality for log-concave probability measures, which 
  was proved by Brascamp and Lieb in \cite{BL}. Its main feature is the weight in the energy term.
It is valid for functions on $\R^n$ but we will apply it only for $n=1$ where the statement is as follows.

Let $\tau$ be a probability measure on $\R$, with density $\ee^{-g}$ such that $g\in C^2(\R)$ and 
$g''>0$ in $\R$. Then for every $C^1(\R)$ function $u$ with zero mean and finite variance with respect to $\tau$, it holds
\begin{equation}\label{BrLieb}
\int_{\R}u^2 \ee^{-g}\,dx \le \int_{\R} u'^2 \frac1{g''}\ee^{-g}\, dx.  
\end{equation}

%%%%%%%%%%%%---------------------------------------------------------------------------------------------------------------------------
\subsection{Tensorizing weighted Poincar\'e inequalities}
%%%%%%%%%%%%---------------------------------------------------------------------------------------------------------------------------
A classical feature of Poincar\'e inequalities is the so-called tensorization property: for any probability measures $\tau_1$  and 
$\tau_2$ (on $\R^{d_1}$ and $\R^{d_1}$), 
it holds $\lambda_{\tau_1\times\tau_2}=\min(\lambda_{\tau_1},\lambda_{\tau_2})$ where $\tau_1\times\tau_2$ is the product measure.
The study of the stability conditions will naturally lead to weighted Poincar\'e inequalities, for which tensorization issues are
more complicated, as the next statement shows.  

%%%%%%%%%%%%%%%%%
\begin{theorem}\label{teoPoincpeso}
%%%%%%%%%%%%%%%%%
Let $\tau,\nu$ be two smooth probability measures on $\R$ with density: $d\tau(y)=\ee^{\psi(y)}dy$, $d\nu(y)=\ee^{\fii(y)} dy$ and let $\Theta:\R\to (0,+\infty)$.
Then for every $u\in C^{\infty}_0(\R^{M+1})$,
% \begin{equation}\label{ptheta}\tag{$P_\Theta$}
%  \int_{\R^{M+1}}u d\tau^M(x)d\nu(y)=0 \text{ implies }
%  \int_{\R^{M+1}}u^2 \Theta(y)d\tau^M(x)d\nu(y) \le \int_{\R^{M+1}}|Du|^2d\tau^M(x) d\nu(y),
% \end{equation}
\begin{align}\label{ptheta}\tag{$P_\Theta$}
 \int_{\R^{M+1}}u(x,y) d\tau^M(x)d\nu(y)=0 \qquad\text{ implies }& \\
 \int_{\R^{M+1}}u^2(x,y) \Theta(y)d\tau^M(x)d\nu(y) \le \int_{\R^{M+1}}|Du(x,y)|^2d\tau^M(x) d\nu(y)&, \nonumber
\end{align}
if and only if the following conditions hold:
\begin{enumerate}
\item\label{p1} for every $w\in C^{\infty}_0(\R)$ such that $\int_R w(y)d\nu(y)=0$ we have
$$
\int_{\R}w^2(y)\Theta(y)d\nu(y)\le \int_{\R}w'^2(y)d\nu(y);\\
$$
\item\label{p2} every function $w\in C^{\infty}_0(\R)$ satisfies
$$
\int_{\R}w^2(y)\Theta(y)d\nu(y) \le \la_{\tau} \int_{\R} w^2(y)d\nu(y)+\int_{\R}w'^2(y) d\nu(y).
$$
\end{enumerate}
%%
%%Moreover {if there exists}
%%$$
%%\la=\min_{u\; :\, \int u d\tau^Md\nu=0} \frac{\int_{\R^{M+1}}|Du|^2 d\tau^Md\nu}{\int_{\R}u^2\Theta(y)d\tau^Md\nu},
%%$$
%%then conditions (\ref{ptheta}) and (\ref{p1}),(\ref{p2}) are equivalent.
\end{theorem}

\begin{proof}
We are going to prove the two implications separately. Let us start by the easier one, which actually works
even if $\Theta$ changes signs.

[(\ref{ptheta})$\Longrightarrow$ (\ref{p1}) and (\ref{p2})]
Consider a function $u$ of the form: $u(x,y)=s(x)w(y)$ for $x\in\R^M, y\in\R$, with $s\in C^{\infty}_0(\R^M)$, $w\in C^{\infty}_0(\R)$.
Hence $u$ has zero mean with respect to the measure $d\mu^Md\nu$ if and only if either $\int_{\R^M}s(x)d\tau^M(x)=0$, or $\int_\R w(y)d\nu(y)=0$.
Moreover, as $|Du|^2=w^2|Ds|^2+s^2 w'^2$, and $s\not\equiv 0$, condition (\ref{ptheta}) reads as
\begin{eqnarray}\label{poincpeso1}
\text{ either }\int_{\R^M}s(x)d\tau^M(x)=0, \text{ or } \int_\R w(y)d\nu(y)=0 \text{ implies }\nonumber\\
 \int_{\R}w^2(y)\Theta(y)d\nu(y) \le \dfrac{\int_{\R^M}|Ds(x)|^2d\tau^M(x)}{\int_{\R^M}s(x)^2d\tau^M(x)} \int_\R w^2(y)d\nu(y) +\int_\R w'^2(y)d\nu(y).
\end{eqnarray}
Assume $\int_\R w(y)d\nu(y)=0$ and in (\ref{poincpeso1}) pass to the infimum over all $s(x)\in C^{\infty}_0(\R^M)$.
As
$$
\inf_{u\in C^{\infty}_0(\R^M)}\dfrac{\int_{\R^M}|Ds(x)|^2d\tau^M(x)}{\int_{\R^M}s(x)^2d\tau^M(x)},
$$
can be easily proved to be zero, we get condition (\ref{p1}).
Assume now $\int_{\R^M}s(x)d\tau^M(x)=0$.
Consider (\ref{poincpeso1}) and pass to infimum over $s(x)$ with compact support and zero mean with respect to $d\tau^M$.
By approximation we have
\begin{eqnarray*}
&\inf\left\{ \dfrac{\int_{\R^M}|Ds|^2d\tau^M}{\int_{\R^M} s^2d\tau^M}\ :\ s\in C^{\infty}_0(\R^M),\ \int_{\R^M} s(x)d\tau^M(x)=0 \right\} \\
&=\inf\left\{ \dfrac{\int_{\R^M}|Ds|^2d\tau^M}{\int_{\R^M} s^2d\tau^M}\ :\ \int_{\R^M} s(x)d\tau^M(x)=0 \right\}= \la_{\tau^M}=\la_{\tau},
\end{eqnarray*}
which entails condition (\ref{p2}).
\medskip

%%%%%%%%%%%%%%%%%
%%
[(\ref{p1}) and (\ref{p2})$\Longrightarrow$ (\ref{ptheta})]
Let $\qn{M+1}{k}$ be the $(M+1)$-dimensional cube $[-k,k]^{M+1}$.
We are going to prove the statement in the cubes $\qn{M+1}{k}$, the conclusion in the general case
follows by a standard approximation argument. The advantage of the case of cubes is that,
by the positivity of $\Theta$, we are allowed to use standard compactness results (notice that the same argument would need substantial changes in the case the considered densities vanish at some point). Define
\begin{equation}\label{poincpeso4}
\la_k=\min_{\substack{  v\in C^{\infty}_0(\qn{M+1}{k}) \\ \int v d\mu^M\nu=0  }}\frac{\int_{\qn{M+1}{k}}|Dv|^2 d\tau^Md\nu}{\int_{\qn{M+1}{k}}v^2\Theta(y)d\tau^Md\nu};
\end{equation}
our aim is then to prove $\la_k\ge 1$.
Assume the minimum to be attained by a function $u(x,y)\in C^{\infty}_0(\qn{M+1}{k})$, that is $\int_{\qn{M+1}{k}}ud\mu^Md\nu=0$ and
$$
\frac{\int_{\qn{M+1}{k}}|Du|^2 d\tau^Md\nu}{\int_{\qn{M+1}{k}}u^2\Theta(y)d\tau^Md\nu}=\la_k\,.
$$
Hence $u$ solves the corresponding Euler-Lagrange equation and there exists $\sigma\in\R$ such that for every $x\in \text{int}\qn{M+1}{k}$
\begin{eqnarray*}
\begin{cases}
\Delta_xu+ \langle D\psi;D_xu\rangle + u_{yy}+ \fii'u_y +\la_k\Theta u+\sigma=0,\\
u_{x_1}(\pm k,x_2,...,x_{M},y)\equiv ...\equiv u_{x_M}(x_1,...,x_{M-1},\pm k,y)\equiv u_y(x_1,...,x_M,\pm k)\equiv 0.
\end{cases}
\end{eqnarray*}
Integrating the previous relation on $\qn{M}{k}$ with respect to the measure $d\tau^M(x)$, it follows
$$
\int_{\qn{M}{k}}u_{yy}d\tau^M(x)+ \int_{\qn{M}{k}}u_yd\tau^M(x)\ \fii'(y)+\la_k\Theta(y) \int_{\qn{M}{k}}u\;d\tau^M(x) +\sigma=0,
$$
that is
$$
g''(y)+g'(y)\fii'(y)+\la_k\Theta(y) g(y)+\sigma=0, \qquad\text{ with }g'(\pm k)=0,
$$
where $g(y)=\int_{\qn{M}{k}}u(x,y)d\tau^M(x)$.
Let us integrate the latter equation with respect to $g(y)d\nu(y)$; recalling that $\int_{\qn{1}{k}}g(y)d\nu(y)=\int_{\qn{M+1}{k}}u(x,y)d\tau^M(x)d\nu(y)=0$, we find
\begin{equation}\label{poincpeso3}
\int_{\qn{1}{k}} g'^2d\nu(y)=\la_k \int_{\qn{1}{k}} g^2 \Theta(y)d\nu(y).
\end{equation}
Assume $g\not\equiv 0$.
Then thanks to (\ref{p1}) and (\ref{poincpeso3}), it holds
$$
\la_k=\dfrac{\int_{\qn{1}{k}} g'^2d\nu(y)}{\int_{\qn{1}{k}} g^2 \Theta(y)d\nu(y)}\ge 1,
$$
which implies, recalling (\ref{poincpeso4}), condition (\ref{ptheta}) on $\qn{M+1}{k}$.

On the other hand, let us consider the case $g\equiv 0$, that is $\int_{\qn{M}{k}}u(x,y)d\tau^M(x)=0$ for every $y\in\qn{1}{k}$.
Condition (\ref{p2}), together with an integration with respect to $d\tau^M$ give
\begin{equation}\label{poincpeso2}
\int_{\qn{M+1}{k}} u^2\Theta(y) d\tau^M(x)d\nu(y) \le \la_\tau \int_{\qn{M+1}{k}} u^2 d\tau^M(x)d\nu(y) +\int_{\qn{M+1}{k}} u_y^2 d\tau^M(x)d\nu(y).
\end{equation}
Notice that the Poincar\'e inequality for the measure $\tau^M$ entails
$$
\int_{\qn{M+1}{k}} u^2 d\tau^Md\nu \le \frac 1{\la_{\tau}} \int_{\qn{M+1}{k}}|D_xu|^2d\tau^Md\nu,
$$
and hence by (\ref{poincpeso2}) it follows
\begin{eqnarray*}
 &\int_{\qn{M+1}{k}} u^2\Theta(y) d\tau^M(x)d\nu(y) \le \int_{\qn{M+1}{k}} |D_xu|^2 d\tau^M(x)d\nu(y) +\int_{\qn{M+1}{k}} u_y^2 d\tau^M(x)d\nu(y) \\
        &=\int_{\qn{M+1}{k}} |Du|^2 d\tau^M(x)d\nu(y),
\end{eqnarray*}
that is $\la_k\ge 1$.
This leads to condition (\ref{ptheta}) in the cube $\qn{M+1}{k}$ and, as already noticed, this is enough to prove (\ref{ptheta}) in the whole $\R^{M+1}$.
\end{proof}

%%%%%%%%%%%%%%%%%%%%%%%%%%%%%%%%%%%%%%%%%%%%%%%%%%%%%%%%%%%%%%%%%%%%%%%%
\section{Analysis of half-spaces}\label{secstab}
%%%%%%%%%%%%%%%%%%%%%%%%%%%%%%%%%%%%%%%%%%%%%%%%%%%%%%%%%%%%%%%%%%%%%%%%
%%%%%%%%%%%-----------------------------------------------------------------------------------------------------------------------------
\subsection{Stationary and stable sets}
%%%%%%%%%%%-----------------------------------------------------------------------------------------------------------------------------

We start by recalling the characterization of stationarity and stability, with respect to the isoperimetric inequality, established in \cite{RCBM}, that will be used throughout the rest of the paper. 
In fact we will adopt such characterizations as definitions of stationary and stable sets. 

Let $d\mu(x)=f(x)\,dx=\ee^{\psi}\,dx$ be a smooth probability measure in $\R^{N+1}$ and consider $A\subseteq\R^{N+1}$ an open set with $C^2$ boundary. $A$ is {\em stationary} for the measure $\mu$ if it has constant generalized mean curvature, that is
\begin{equation}\label{stationarity}
\hh_{\psi}(\bd A)= N\hh(x) -\langle D\psi(x),\nu(x) \rangle \Big|_{\bd A}= \text{constant},
\end{equation}
where $\hh$ is the standard mean curvature of $\bd A$ and $\nu(x)$ is its outer unit normal vector at $x$.
Moreover, if for every function $u\in C^{\infty}_0(\bd A)$ such that $\int_{\bd A} u(x)f(x)\,da=0$, it holds
\begin{equation}\label{stability}
\qq{\psi}(u)=\int_{\bd A} f(x)\Big( |D_{\bd A}u(x)|^2-K^2 u^2(x) \Big)\, da(x)+\int_{\bd A} f(x) u^2(x)\left\langle D^2\psi(x)\nu(x);\nu(x)\right\rangle\,da(x)  \ge 0,
\end{equation}
where $K^2$ is the sum of the squared principal curvatures of $\bd A$ and $da(\cdot)$ denotes the element of area, then $A$ is {\em stable}.

%%%%%%%%%%%-----------------------------------------------------------------------------------------------------------------------------
\subsection{Stationarity of half-spaces}
%%%%%%%%%%%-----------------------------------------------------------------------------------------------------------------------------

Let us first introduce the following class of unit vectors
$$
\vn{M}=\Big\{ v\in\sn{M}\ :\  \exists i\neq j\in\{1,...,M\}\text{ s.t. }\vv_k=0\ \text{for }k\not\in\{ i,j\},\ |\vv_i|=|\vv_j|\neq 0\Big\}.
$$
Hence each element of $\vn{M}$ has exactly two non-null coordinates which are in absolute value equal to $\frac 1{\sqrt{2}}$.
Let us indicate by $\vn{M}^+$ the subclass of $\vn{M}$ such that the non-null components have the same sign (end hence $\vv_i=\vv_j$), and by $\vn{M}^-$ the subclass with components of different signs (that is $\vv_i=-\vv_j$).
In particular we define $\tilde{\vv}\in\vn{N}$ such that $\tilde{\vv}_{N+1},\tilde{\vv}_{N}\neq 0$, hence $\tilde{\vv}\in\vn{N}^+$ if $\tilde{\vv}=(0,...,0,1/{\sqrt{2}},1/{\sqrt{2}})$ or $\tilde{\vv}=(0,...,0,-1/{\sqrt{2}},-1/{\sqrt{2}})$; while $\tilde{\vv}\in\vn{N}^-$ if $\tilde{\vv}= (0,...,0,-1/{\sqrt{2}},1/{\sqrt{2}})$ or $\tilde{\vv}=(0,...,0,1/{\sqrt{2}},-1/{\sqrt{2}})$.

A characterization of the stationarity of half-spaces for general product measures follows.
%%%%%%%%%%%%%%%%%%%%%%%%%%%%%%%
\begin{theorem}\label{stazteo}
%%%%%%%%%%%%%%%%%%%%%%%%%%%%%%%
Let $\mu_i$ be probability measures on $\R$: $d\mu_i(t)=\ee^{\psi_i(t)}dt$, $t\in\R$, with $\psi_i\in C^2(\R)$, for $i=1,...,N+1$.
Consider their product measure $d\mu(x)=d\mu_1(x_1)\cdot ... \cdot d\mu_{N+1}(x_{N+1})$, with  $x=(x_1,...,x_{N+1})\in\R^{N+1}$.
Let $H^{N+1}_{\vv,t}$ be the half-space of $\R^{N+1}$:
$$
H^{N+1}_{\vv,t}=\left\{\ x\in\R^{N+1}\ :\ \langle x,\vv \rangle < t\right\}=\left\{x\in\R^{N+1}\ :\ \sum_{i=1}^{N+1} x_i\vv_i< t\right\},
$$
with $\vv\in \sn{N}$.
$H^{N+1}_{\vv,t}$ is stationary if and only if at least one of the following holds:
\begin{enumerate}[(i)]
\item \label{1} $H^{N+1}_{\vv,t}$ is a coordinate half-space;
\item \label{2} $\vv$ has (exactly) two non-null components $\vv_i,\vv_j$ and $\psi_i''(x)=\psi''_j(\tau-\alpha x)$, where $\tau= \frac{t}{\vv_j}$, and $\alpha=\frac{\vv_i}{\vv_j}$;
\item \label{3} $\vv$ has at least three non-null components $\vv_i,\vv_j,\vv_k$ and the corresponding measures $\mu_i,\mu_j,\mu_k$ are Gaussian with the same variance.
\end{enumerate}
\end{theorem}
%%%-----------------------------------------
\begin{proof}
%%%-------------------------------------------
Up to a rearrangement of variables, we may assume $\vv_{N+1}\neq 0$.
Condition (\ref{stationarity}) for the half-space $H^{N+1}_{\vv,t}$  and the measure $\mu^{N+1}$, reads as
\begin{equation}\label{hcostante}
\sum_{i=1}^{N}\psi_i'(x_i)\; \alpha_i + \psi_{N+1}'\Big(\tau-\sum_{i=1}^{N}\alpha_i\; x_i\Big) = \text{constant},       
\end{equation}
for every $x_i\in\R$, where $\tau=\frac {t}{\vv_{N+1}}$ and $\alpha_i=\frac{\vv_i}{\vv_{N+1}}$.

[$\Longrightarrow$]

Notice that the left hand side in condition (\ref{hcostante}) can be seen as a function of $N$-variables $x_1,...,x_{N}$, and hence we can differentiate it with respect to each variable $x_i$ getting
\begin{equation}\label{psi2}
\alpha_k\; \psi_k''(x_k)-\alpha_k\; \psi_{N+1}''\Big(\tau-\sum_{i=1}^{N}\alpha_ix_i\Big) = 0, \qquad k=1,...,N.
\end{equation}
One of the three following situations happens:
\begin{enumerate}[(a)]
\item \label{a} $\alpha_k=0$ for $k=1,...,N$;
\item \label{b} there exists only one $k\in\{1,...,N\}$ such that $\alpha_k\neq 0$;     
\item \label{c} there exist at least two non-null $\alpha_k,\alpha_j$.
\end{enumerate}

Case (\ref{a}) corresponds to the case $H^{N+1}_{\vv,t}$ coordinate, which is condition (\ref{1}).
In case (\ref{b}) conditions (\ref{psi2}) becomes $\psi_k''(x)=\psi_{N+1}''(\tau-\alpha_k x)$ for every $x\in\R$, that is, condition (\ref{2}). In case (\ref{c}) equation (\ref{psi2}) is
$$
\psi_k''(x)= \psi_{N+1}''\Big(\tau-\sum_{i=1}^{N}\alpha_ix_i\Big),
$$
where on the left hand side there is a one-variable function, while the right hand side depends on at least two variables.
This entails that $\psi''_k$ has to be constant, for every $k$ such that $\alpha_k\neq 0$, that is $\psi''_k\equiv \psi''_{N+1}\equiv \text{ constant }$, which is condition (\ref{3}).

[$\Longleftarrow$]

Case (\ref{1}) can be proved by trivial calculation.
Let us consider case (\ref{2}).
As $\psi_i''(x)=\psi''_{N+1}(\tau-\alpha x)$, integrating with respect to $x\in\R$ we get $\alpha\; \psi_i'(x)+\psi_{N+1}'(\tau-\alpha x)=\psi_{N+1}'(\tau)+\alpha\;\psi_i'(0)$, which is equivalent to
$$
\sum_{i=1}^{N}\psi_i'(x_i)\;\alpha_i + \psi_{N+1}'\Big(\tau-\sum_{i=1}^{N}\alpha_i\;x_i\Big) = \psi_{N+1}'(\tau)+\alpha\; \psi_i'(0),
$$
and hence (\ref{hcostante}) holds.

Consider case (iii).
For each non-null $\alpha_i$ we have
$$
\psi'_i(x_i)=-\frac{(x_i-m_i)}{\sigma^2};
$$
hence the left hand side of condition (\ref{hcostante}) becomes
$$
-\frac 1{\sigma^2}\sum_{i=1}^{N+1}x_i\alpha_i +\frac 1{\sigma^2}\sum_{i=1}^{N+1}m_i\alpha_i -\frac 1{\sigma^2}\Big( \tau-\sum_{i=1}^{N+1}\alpha_ix_i-m_{N+1} \Big),
$$
which is constant.
\end{proof}

Particularly relevant is the case of $N$-fold product measures where directions of possible stationary half-spaces are of only three types.
%%%%%%%%%%%%%%%%%%
\begin{corollary}\label{stazcoroll}
%%%%%%%%%%%%%%%%%
Let $\mu$ be a measure on $\R$: $d\mu(t)=\ee^{\psi(t)}dt$, $t\in\R$, with $\psi\in C^2(\R)$.
Consider its product measure $d\mu^{N+1}(x)$ in $\R^{N+1}$.
$H^{N+1}_{\vv,t}$ is stationary if and only if at least one of the following holds:
\begin{enumerate}[(i)]
\item\label{31} $H^{N+1}_{\vv,t}$ is a coordinate half-space;
\item\label{32} $\vv\in\vn{N}^-$ and $\psi''$ is $\tau=\sqrt{2}t$-periodic;
\item\label{33} $\vv\in\vn{N}^+$ and $\psi''$ is symmetric with respect to $\frac{\tau}2=\pm\frac{\sqrt{2}t}2$;
\item\label{34} $\mu$ is Gaussian: $d\mu(x)=\frac 1{\sqrt{2\pi\sigma^2}}\ee^{-\frac{x^2}{2\sigma^2}}$.
\end{enumerate}
\end{corollary}
%%%%%%%%%%%%%%%%%
\begin{remark}
%%%%%%%%%%%%%%%%%
Notice that the half-space $H^{N+1}_{{\vv},0}$, with $\vv\in\vn{N+1}^-$, is always stationary for $\mu^{N+1}$.
Furthermore, if the measure $\mu$ is symmetric, then the half-space $H^{N+1}_{{\vv},0}$ is  stationary, for every $\vv\in\vn{N+1}$.
\end{remark}
%%--------------------
\begin{proof}[Proof of Corollary \ref{stazcoroll}]
Thanks to Theorem \ref{stazteo} it is enough to prove that, if $H^{N+1}_{{\vv},t}$ is stable and $\vv$ has two non-null components $\vv_k,\vv_j$, then either $|v_k|=|v_j|$ (that is $\vv\in\vn{N+1}$), or $\mu$ is Gaussian.   
Thanks to condition (\ref{psi2}) we have
$$
 \psi''(x)= \psi''\Big(\tau-\alpha_ix\Big),\qquad\text{ for }i=k,j,
$$
for every $x\in\R$.
If $|\alpha_k|=1$, that is $\vv\in\vn{N}^{\pm}$, then by Theorem \ref{stazteo} conditions (\ref{33}) and (\ref{32}) follow.

On the other hand let us assume $|\alpha_k|=|\frac{\vv_k}{\vv_j}|\neq 1$ (by a change of variables we may assume $|\alpha_k|< 1$).
Iterating this relation, we get
$$
\psi''(x)=\psi''\Big(\tau\sum_{i=0}^n (-1)^i\alpha_k^i- \alpha_k^{n+1}\;x\Big),
$$
and, considering the limit as $n$ tends to infinity, it follows that $\psi''$ is constant and hence $\mu$ is Gaussian.
\end{proof}

%%-------------------
%%

%%%%%%%%%%%%%%%%%%%%%%%%%%%%%%%%%%%%%%%%%%%%%%%%%%%%%%%%%%%%%%%%%%%%%%%%%%%%%%%%%%%%%%%%%%%%%%%%%%%%%%%%%%%%%%%%%%%%%%%%%
\subsection{Stability of half-spaces}
%%%%%%%%%%%%%%%%%%%%%%%%%%%%%%%%%%%%%%%%%%%%%%%%%%%%%%%%%%%%%%%%%%%%%%%%%%%%%%%%%%%%%%%%%%%%%%%%%%%%%%%%%%%%%%%%%%%%%%%%%
Let us now study the stability of half-spaces for the $N$-fold product measures $\tau^N$.
Since stationarity is a necessary condition for  stability, we only have to analyse the cases which are mentioned in  Corollary \ref{stazcoroll}.
In particular we focus on the analysis to the case of log-concave measures.

We are going to prove that stability is strictly related to Poincar\'e type inequalities.
In particular in Theorem \ref{stabteocoordinate} we prove that the coordinate half space $\{x_{N}<t\}$ is stable for the $N$-product measure $\mu^N$ if $-\psi''(t)\le \la_\mu$ where $d\mu(x)=\ee^{\psi(x)}dx$.
Moreover in Theorem \ref{stabteononcoordinate} the non-coordinate case is treated. 
The $N$-dimensional problem in fact reduced to the 3-dimensional case and this follows by the tensorization of Poincar\'e inequalities with weights. 

We split the analysis in the two cases, which are treated in Section \ref{seccoordinate} and in Section \ref{secnoncoordinate} respectively.

%%%%%%%%%%%-----------------------------------------------------------------------------------------------------------------------------
\subsection{Stability of coordinate half-spaces}\label{seccoordinate}
%%%%%%%%%%%-----------------------------------------------------------------------------------------------------------------------------

%%%%%%%%%%%%%%%%%%%%%%%%%%%%
\begin{theorem}\label{stabteocoordinate}
%%%%%%%%%%%%%%%%%%%%%%%%%%%%
Let $\mu$ be a probability measure in $\R$, with support the whole real line: $d\mu(t)=\ee^{\psi(t)}dt$, $t\in\R$, with $\psi\in C^2(\R)$ and let  $\ee^{\phi(x)}$ be the density of its $(N+1)$-dimensional product measure. 
Denote by $\lambda_\mu$ the spectral gap of $\mu$.
For $t\in\R$,  let
$$
H^{N+1}_t=\left\{\ x\in\R^{N+1}\ :\ x_{N+1} < t\right\}.
$$
The half space $H^{N+1}_t$ is stable if and only if $-\psi''(t)\le \la_{\mu}$.
\end{theorem}
%%---------------
\begin{proof}
%%--------------
By condition (\ref{stability}), stability of $H_{t}^{N+1}$ is equivalent to $\qq_{\psi}(u)\ge 0$, 
for each function $u\in C^{\infty}_0(\R^N)$ such that 
\begin{equation}\label{0mean2}
\int_{\rn} u(x_1,...,x_N,t)\, d\mu^N(x_1,...,x_N)=0,   
\end{equation}
where
\begin{equation*}
\qq_{\psi}(u)= \int_{\R^N} \Big( |Du(x_1,...,x_N,t)|^2+u^2(x_1,...,x_N,t) \psi''(t) \Big) f(x_1)\cdot ... \cdot f(x_{N}) f(t)\ dx_1\cdots dx_{N}.
\end{equation*}  
Hence $H$ is stable if and only if for every function $u\in C^{\infty}_0(\rn)$ satisfying (\ref{0mean2}), it holds
$$
\int_{\R^{N}} |Du|^2 d{\mu^{N}} \ge -\psi''(t)\int_{\R^{N}} u^2 d{\mu^{N}}.
$$
Notice that this is a Poincar\'e inequality, hence it holds if and only if $-\psi''(t)\le \la_{\mu^N}$, and consequently, by the tensorization property of the \poinc inequality, if and only if $-\psi''(t)\le{\la_{{\mu}}}$.
\end{proof}

Notice that previous result is trivial in the case of Gaussian measures. 
Indeed in that case each half space is isoperimetric and hence stable.
Moreover the reverse holds true: a symmetric probability measure $\mu$ such that  all coordinate half-spaces are isoperimetric regions for $\mu^N$ is necessarily Gaussian (see \cite{BH2,KPS,O}).
We show next that the stability of all coordinate half-spaces also characterizes Gaussian measures.
%%%%%%%%%%%%%%%%
\begin{theorem}\label{allcoord}
%%%%%%%%%%%%%%%
Let $\mu$ be a probability measure on $\R$, with $d\mu(x)=\ee^{-v(x)}\,dx$, $v\in C^2(\R)$.
Consider its product measure $\mun$, $N\ge 2$.
If for every $t\in \R$ the coordinate half-space
$$
H_t=\big\{x\in\rn\ :\ x_N<t\big\},
$$      
is stable for $\mun$, then $v$ is quadratic, that is $\mu$ is Gaussian.
\end{theorem}

\begin{proof}
Let us point out that necessarily $\sup_{x\in\R}v''(x)>0$, as otherwise $-v$ would be a convex function and hence $\ee^{-v}$ could not be a probability density on $\R$.

Let us now consider the stability conditions for the coordinate half space $H_t$.
By (\ref{stability}), $H_t$ is stable if and only if for every $u\in C^{\infty}_0(\R^{N-1})$ such that $\int_{\R^{N-1}}u\,d\mu^{N-1}=0$ we have
$$
v''(t)\int_{\R^{N-1}}u^2 \, d\mu^{N-1} \le \int_{\R^{N-1}}|Du|^2 \, d\mu^{N-1}.
$$      
In particular it holds
$$
\int_{\R^{N-1}}u^2 \, d\mu^{N-1} \le \frac 1{\sup v''}\int_{\R^{N-1}}|Du|^2 \, d\mu^{N-1}.
$$
Using a density argument, we can apply the previous inequality to the one-variable function $u(x)=x$, seen as a function of $(N-1)$ variables, and we get
\begin{equation}\label{x}
\int_{\R}x^2 \, d\mu \le \frac 1{\sup v''}.     
\end{equation}
Using this last estimate, together with two integrations by parts and H\"older's inequality, we get
\begin{eqnarray*}
1&=& \int_{\R}xv'\ee^{-v}\,dx\le \left(\int_{\R}x^2\,d\mu\right)^{\frac 12}\left(\int_{\R}v'^2\, d\mu\right)^{\frac 12}\\
               &\le& \frac 1{(\sup v'')^{\frac 12}} \left( \int_{\R}v'^2\, d\mu \right)^{\frac 12}= \frac 1{(\sup v'')^{\frac 12}} \left(\int_{\R} v''\, d\mu\right)^{\frac 12},
\end{eqnarray*}
which gives
$$
\sup v'' \le \int_{\R} v''\, d\mu,
$$
and hence, thanks to the regularity of $v''$, $v''\equiv \sup v''$. 
Therefore $v''$ is constant that is  $\mu$ is Gaussian.
\end{proof}

%%%%%%%%%%%%%%%
\begin{remark}\label{la<v2}
%%%%%%%%%%%%%%%
The above argument actually shows that (provided $v''$ is $\mu$-integrable),
$$
\lambda_\mu\le \int_\R v''(x)\, d\mu(x).
$$
This is a well-known fact, for which we have given a naive proof for sake of completeness.
The conceptual proof consists in using the fact, related to the so-called $L^2$-method of H\"ormander, that $\lambda_\mu$ is also the biggest constant such that for all $u$,
$$
\lambda_\mu \int_\R u'^2(x) d\mu(x) \le \int_\R \big(u''^2(x)+v''(x) u'^2(x)\big)\, d\mu(x).
$$
It is then clear that
\begin{equation}\label{infv2}
\inf v'' \le \lambda_\mu \le \int_\R v''(x) \,d\mu(x),
\end{equation}
where the upper bound follows from approximations of  $u(x)=x$.
\end{remark}

%%%%%%%%%%%%%%%
The previous result can be used to establish the existence of stable coordinate half-spaces for a measure $d\mu(x)=\ee^{-v(x)}dx$ with $v\in C^2(\R)$.
More precisely the following holds.
%%%%%%%%%%%%%%%%%%%
\begin{proposition}
%%%%%%%%%%%%%%%%%%%
Let $\mu$ be a probability measure on $\R$ with density $\ee^{-v(x)}$, $v\in C^2(\R)$, and consider its product measure $\mu^{N+1}$.
There exists at least a real number $t$ such that the coordinate half-space $H_t=\{x_{N+1}\le t\}$ is stable for the measure $\mu^{N+1}$.
\end{proposition}
%%%-------------
\begin{proof}
As shown in Theorem \ref{stabteocoordinate} a sufficient condition for the stability of $H_t$ is $v''(t)\le \la_\mu$, where $\la_\mu$ is the best constant in the Poincar\'e inequality and hence it is well defined and non-negative.
We want to show that inequality (\ref{infv2}) implies the existence of at least on stable coordinate half-space.

Notice that this is obvious in the case $\inf_{x\in\R}v''(x)<0$; hence we may assume $\inf v''\ge 0$ that is, $\mu$ log-concave.
Consequently $\la_\mu>0$ (see \cite{B2}) whence the case $\inf v'' =0$ follows immediately.
Finally we assume
\begin{equation}\label{c>0}
\inf_{x\in\R} v''(x)= c>0.
\end{equation}
Moreover by (\ref{infv2}) we may assume $\la_\mu=\inf v''$, as the case $\inf v''< \la_\mu$ trivially implies the assert of the proposition. In other words we have
\begin{equation}\label{28}
\inf v''=\inf_{\substack{u\in\wud{\mu}(\R)\\\int_\R ud\mu=0}} \dfrac{\int_\R u'^2\;d\mu}{\int_\R u^2d\mu}=c>0,
\end{equation}
and we are going to show that this latter case occurs only for the Gaussian measure.
We first notice that the infimum in (\ref{28}) is attained at some function $u$ (Indeed thanks to (\ref{c>0}) and the Bakry-Emery criterion 
\cite{BE}, $\mu$ satisfies 
a logarithmic Sobolev inequality. Hence the work of Wang \cite{W} ensures that the operator $L$ defined by $Lf=f''-v'f'$ has an empty essential spectrum. It follows easily that it has a pure point spectrum and one my choose $u$ as an eigenfunction for the first non-zero eigenvalue of $-L$ ).
Consequently
$$
\inf v''=\displaystyle \dfrac{\int_\R u'^2\;d\mu}{\int_\R u^2d\mu}\ge \  \dfrac{\inf v''\;\displaystyle\int_\R{\displaystyle \dfrac{u'^2}{v''}\;d\mu}}{\int_\R u^2d\mu}\ge \inf v'',
$$
where the last inequality follows from the Brascamp-Lieb inequality (\ref{BrLieb}).
This entails
\begin{equation}\label{v2C}
\int_\R u'^2d\mu=\inf v''\; \int_\R \frac{u'^2}{v''}d\mu.
\end{equation}
We claim that
$$
\mu(\{x\in\R\ :\ u'(x)=0\}) = 0,
$$
which implies, by (\ref{v2C}), $v''(x)=\inf v''$ $\mu$-almost everywhere and then, by the regularity of $v$, that $v''$ is constant, that is the measure $\mu$ is Gaussian.

Indeed, let us define $a,b\in\mathbb{R}\cup\{\pm\infty\}$ as
$$
a=\inf\{x\in\R\;:\;u'(x)=0\},\qquad b=\sup\{x\in\R\;:\;u'(x)=0\},
$$
and let $\{a_k\},\{b_k\}$ be approximating sequences such that $a_k$ converges to $a$ and $b_k$ converges to $b$ as $k$ tends to infinity, with $u'(a_k)=u'(b_k)=0$ for every $k\in\N$.
By classical results for Sturm-Liouville problems (see \cite{H}), applied to the Euler-Lagrange equation associated to the minimum Problem (\ref{28}), the set $\{x\in\R\ :\ u(x)=0\}\cap(a_k,b_k),$ is finite for every $k\in\N$.
This implies in particular that $U_k=\{x\in\R\;:\; u'(x)=0\}\cap(a_k,b_k)$ has zero measure, with respect to $\mu$,
for every $k$. Indeed, we show that $U_k$ only has a finite number of accumulation points.
Assume $U_k$ admits an accumulation point $\bar x$, then $u'(\bar x)=0$ and $u''(\bar x)=0$, since we can choose a sequence $x_k\in U_k$, with $x_k\to \bar x$, and it holds
$$
\frac{u'(\bar x)-u'(x_k)}{\bar x-x_k}=0, \qquad\text{for every }k\in\N.
$$
By the Euler-Lagrange equation this implies that also $u(\bar x)=0$ and hence $U_k$ has only a finite number of accumulation points, for every $k\in\N$.
This shows in particular that each $U_k$ has zero $\mu$-measure which gives $\mu(\{u'=0\})=0$.
%
% Finally we have proved that for every non-Gaussian measure $\mu$ it cannot happen $\la_\mu=\inf v''>0$ and hence, by (\ref{infv2}), $\la_\mu>\inf v''$ which implies that there exists $t\in\R$ such that $v''(t)\le \la_\mu$, and then $H_t$ is stable.
\end{proof}

%%%%%%%%%%%-----------------------------------------------------------------------------------------------------------------------------
\subsection{Stability of non-coordinate half-spaces}\label{secnoncoordinate}
%%%%%%%%%%%-----------------------------------------------------------------------------------------------------------------------------

%%%%%%%%%%%%%%%%%%%%%%%%%%%%
\begin{theorem}\label{stabteononcoordinate}
%%%%%%%%%%%%%%%%%%%%%%%%%%%%
Let $\mu$ be an even log-concave measure in $\R$, with support the whole real line: $d\mu(t)=\ee^{\psi(t)}dt$, $t\in\R$, with $\psi\in C^2(\R)$, $\psi''<0$ in $\R$, and let  $\ee^{\phi(x)}$ be the density of its $(N+1)$-dimensional product measure. Moreover, denote by $\lambda_\mu$ the spectral gap of $\mu$.
For $\vv\in \sn{N}$ and $t\in\R$,  let
$$
H^{N+1}_{\vv,t}=\left\{\ x\in\R^{N+1}\ :\ \langle x,\vv \rangle < t\right\}=\left\{x\in\R^{N+1}\ :\ \sum_{i=1}^{N+1} x_i\vv_i< t\right\}.
$$
If $\mu$ is not Gaussian and $\vv\not\in\vn{N}$, then $H^{N+1}_{\vv,t}$ is not stable.

For $\vv\in\vn{N}$, the half space $H^{N+1}_{{\vv},t}$ is stable if and only if so is $H^3_{{\vv},t}$.
Moreover if $H^3_{{\vv},t}$ is stable, then so is $H^2_{{\vv},t}$.
\end{theorem}
%%---------------
\begin{proof}
%%--------------
Notice that, as stationarity is a necessary condition for stability, by Theorem \ref{stazteo} we have that  $H^{N+1}_{\vv,t}$ can be stable only if $\vv\in\vn{N}$.

By condition (\ref{stability}), stability of $H_{\vv,t}^{N+1}$ is equivalent to
\begin{equation}\label{Qge0}
\qq_{\psi}(u)= \int_{\bd H_{\vv,t}^{N+1}} \left( |Du|^2+u^2 \langle D^2\phi\, \vv;\vv\rangle) \right) \ff\,\ d x\ge 0,
\end{equation}
for every function $u\in C^{\infty}_0(\bd  H_{\vv,t}^{N+1})$ such that
\begin{equation}\label{0mean}
\int_{\bd  H_{\vv,t}^{N+1}} u\, \ff\, dx=0,
\end{equation}
where  $\ff=\ee^{\phi(x)}$ is the density of $\mu^{N+1}$, that is $\ff(x_1,...,x_N)=f(x_1)\cdots f(x_{N+1})$.
By definition of product measure, $D^2\phi(x)_{ij}=\delta_{ij}\psi''(x_i)$, for every $x=(x_1,...,x_{N+1})\in\R^{N+1}$.

We define $\tau=\sqrt{2}t$ and $\pp=\vv_N/\vv_{N+1}$; hence $\pp\in\{\pm 1\}$.
By the symmetry of $\mu$, it is enough to prove the statement for the half-spaces $H^{N+1}_{\tilde{\vv},t}$ and $H^3_{\tilde{\vv},t}$ with $\tilde{\vv}$ either in $\vn{N}^+$ (corresponding to $\pp=1$) or in $\vn{N}^-$ (case $\pp=-1$).
We are going to show that the $(N+1)$ dimensional problem is equivalent to the case $(N+1)=3$.
In order to apply condition (\ref{Qge0}) we choose the following volume preserving parametrization of $\bd  H_{\vv,t}^{N+1}$. Let $p:\R^{N}\to\R^{N+1}$ be defined as
$$
p(x_1,...,x_N)= (x_1,...,\frac{x_N}{\sqrt{2}},\tau-\pp\frac{x_N}{\sqrt{2}});
$$
hence $p(\R^{N})=\bd  H_{\vv,t}^{N+1}$ and
$$
\left| \frac{\partial p}{\partial x_1} \wedge ... \wedge \frac{\partial p}{\partial x_N}\right|\equiv 1.
$$
Then by (\ref{Qge0}),(\ref{0mean}), $H^{N+1}_{\tilde{\vv},0}$ is stable if and only if
\begin{equation}\label{stabv}
\int_{\rn} \left( \left|Du(x_1,...,x_N)\right|^2+u^2(x_1,...,x_N)\, \psi''(x_{N}/\sqrt{2}) \right) \fg(x_1,...,x_{N})\ dx \ge 0,
\end{equation}
for every function $u(x)=u(x_1,...,x_N)\in C^{\infty}_0(\rn)$, such that
$$
\int_{\rn} u(x)\fg(x_1,...,x_N)\,dx_1...dx_N =0,
$$
where
$$
\fg(x_1,...,x_N)=f(x_1)\cdots f(x_{N-1})\,f(x_{N}/\sqrt{2})\,f(\tau-\pp x_N/\sqrt{2}).
$$
Define the measure $d\mu_{\pp}(x)=f(\tau-\pp x/\sqrt{2})dx$. The previous stability condition reads as a Poincar\'e type inequality for the product measure $d\mu^{N-1}(x_1,...,x_{N-1})\cdot d\mu_{\pp}(x_N)$ with the additional weight $(-\psi''(x_N/\sqrt{2}))$.
More precisely $H^{N+1}_{\tilde{\vv},t}$ is stable if and only if
$$
\int_{\rn} u(x,y)\ d\mu^{N-1}(x_1,...,x_{N-1}) d\mu_{\pp}(y)=0,
$$
implies
$$
\int_{\rn}  \left|Du(x,y)\right|^2   d\mu^{N-1}(x) d\mu_{\pp}(y) \ge \int_{\rn} u^2(x)\big(-\psi''(y/\sqrt{2})\big) d\mu^{N-1}(x) d\mu_{\pp}(y).
$$
In the case $N-1\ge 1$ this is condition (\ref{ptheta}) of Theorem \ref{teoPoincpeso}, applied to the measures $\mu^{N-1}, \mu_{\pp}$ with weight $\Theta(y)=(-\psi''(y/\sqrt{2}))$.
Therefore for $N+1\ge 3$ the stability of $H^{N+1}_{\tilde{\vv},t}$ is equivalent to (\ref{ptheta}) and hence, by Theorem \ref{teoPoincpeso}, it is also equivalent to the following: for every $v\in C^{\infty}_0(\R)$
\begin{eqnarray}
\int_{\R}v(y) d\mu_{\pp}(y)=0\text{ implies } \int_\R v'^2(y) d\mu_{\pp}(y)\ge \int_\R v^2(y)(-\psi''(y/\sqrt{2})) d\mu_{\pp}(y) \label{stab1}\\
%%\text{ for every } v:\R\to\R  \nonumber\\
\text{ and }\int_\R v^2(y) (-\psi''(y/\sqrt{2})) d\mu_{\pp} \le \la_{\mu_{\pp}}\, \int_\R v^2(y)  d\mu_{\pp}(y)+ \int_\R v'^2(y) d\mu_{\pp}(y).\label{stab2}
\end{eqnarray}
Hence, as conditions (\ref{stab1}) and (\ref{stab2}) do not depend on $N$, the first assertion is proved.

Let us now analyse the case $N+1=2$.
The 2-dimensional half-space $\{x_1+\pp x_2\le\tau\}$ is stable for $\mu^2$ if and only if
$$
\int_{-\infty}^{+\infty} v^2(y)(-\psi''(y/\sqrt{2})) d\mu_{\pp}(y)\le \int_{-\infty}^{+\infty} v'^2(y)d\mu_{\pp}(y),
$$
for every function $v$ such that $\int_{-\infty}^{+\infty} v(y)d\mu_{\pp}(y)=0$, that is condition (\ref{stab1}), and hence the theorem is proved.
\end{proof}
%%%%%%%%%%%%%%%%%%%%%%%%%%%
%\begin{remark}
%%%%%%%%%%%%%%%%%%%%%%%%%%%
%In order to characterize the coordinate half-spaces which are stationary, it could be useful to find a practical condition on a generic measure $\mu$ giving $-\psi''(t)\le \la_{\mu}$.
%It could be natural to ask $\max(-\psi'')\le \la_{\mu}$, in such a way that condition (i) is guaranteed.
%Unfortunately, thanks to Theorem \ref{allcoord} below, this implies that $\mu$ is Gaussian.
%\end{remark}

%%%%%%%%%%%%%%%%%%%%%%%%%%%%%%%%%%%%%%%%%%%%%%%%%%%%%%%%%%
\section{Examples}\label{sec:ex}
%%%%%%%%%%%%%%%%%%%%%%%%%%%%%%%%%%%%%%%%%%%%%%%%%%%%%%%%%
In this section, we analyse the stability of hyperplanes for three classes of examples, showing that in the
classification given in the previous sections all possibilities may occur.
For the first set of examples, the only stable half-spaces are coordinate half-spaces. The second example is given by the logistic distribution; in this case the only non-coordinate stable hyperplanes are bisector lines in two dimensions. As this specific example
enjoys many remarkable properties, we are able to push the study a bit further. % and give some estimates of the infinite dimensional isoperimetric profile.
Eventually, we give a third set of examples: namely non-Gaussian measures for which stable non-coordinate hyperplanes exist in any dimensions.

%%%%%%%%%%%%%%%%%%%%%%%%%%%%%%%%%%%%%%%%%%%%%%%%%%%%%%%%%%%%%%%%%%%%%%%%%%%%%%%%%%%%%%%%%%%%%%%%%%%%%%%%%%%%%%%%%%%%%%%%%%%%%%%%%%%%
\subsection{Measures with no non-coordinate stable half-spaces}
%%%%%%%%%%%%%%%%%%%%%%%%%%%%%%%%%%%%%%%%%%%%%%%%%%%%%%%%%%%%%%%%%%%%%%%%%%%%%%%%%%%%%%%%%%%%%%%%%%%%%%%%%%%%%%%%%%%%%%%%%%%%%%%%%%%%
Let us consider a non-Gaussian probability measure on the real line $d\mu(t)=e^{\psi(t)}dt$ where $\psi$ is twice continuously differentiable.
For simplicity we will also assume that $\psi$ and $\psi''$ are even and have no other point of symmetry that 0
(in particular $\psi''$ is not periodic). This ensures that the only non-coordinate stationary hyperplanes for $\mu^2$
contain the origin and are orthogonal to a vector in $\vn{1}$.
Consider the measure
$$
d\nu(t)= e^{2\psi(t/\sqrt2)} dt =e^{-\varphi(t)} dt.
$$
The stability of the bisector lines $x_1=\pm x_2$ for $\nu$
is equivalent to the following functional inequality: for all $u:\R\to \R$ with $\int_\R u\,  d\nu=0$,
\begin{equation}\label{eq:stab2d}
 \int_\R u^2(x) \varphi''(x) \, d\nu(x) \le \int_\R u'^2(x) d\nu(x).
\end{equation}
Note that when $\psi$ is convex, i.e. $\mu$ is a log-concave probability measure, then $\nu$ is also a finite log-concave
measure and Brascamp and Lieb inequality (\ref{BrLieb}) assures that for all functions
 $u:\R\to \R$ with $\int_\R u(x)\,  d\nu(x)=0$,
\begin{equation}\label{eq:BL}
 \int_\R u^2(x) d\nu(x) \le \int_\R \frac{u'^2(x)}{\varphi''(x)}\,  d\nu(x).
\end{equation}
Despite of the striking similarity between \eqref{eq:stab2d} and \eqref{eq:BL}, the former may be false for log-concave finite measures.
This is easily seen with power type potentials $d\nu_p(x)=\ee^{-|x|^p}dx$ for $p>2$.
Indeed the function $u(x)=x\in\wud{\nu_p}(\R)$ does not satisfy inequality (\ref{eq:stab2d}) for the measure $\mu=\nu_p$.
More precisely the following holds.
%%%%%%%%%%%%%%%%%%%
\begin{proposition}
%%%%%%%%%%%%%%%%%%%
Consider the measure $\nu_p$ with $d\nu_p(x)=\ee^{-|x|^p}dx$, $p>2$, and let $\nu_p^N$ be its $N$-dimensional product measure. Each possible stable half-space for $\nu_p^N$ is necessarily a coordinate one.
\end{proposition}
%%--------------
\begin{proof}
Notice that, thanks to Theorem \ref{stazteo} the only possible stable non-coordinate half-space is $H^N_{\vv,0}$ with $\vv\in \vn{N}^{\pm}$.
We are going to show that $H^2_{\vv,0}=\{x_1\pm x_{2}\le 0 \}$ is not stable and hence by Theorem \ref{stabteononcoordinate} the conclusion of the proposition follows.
By the stability conditions (\ref{stabv}), we have that $H^2_{\vv,0}$ is stable for $\nu_p^2$ if and only if for every function $u\in\wud{\nu_p}(\R)$ such that $\int_\R u(x)d\nu_p(x)=0$, inequality (\ref{eq:stab2d}) holds, where $-\varphi(x)=2^{-\frac{p-2}2}|x|^p$.
Consider $u(x)=x\in\wud{\nu_p}(\R)$; as it is an odd function it has zero mean with respect to the measure $\nu_p$.
Moreover it does not satisfy inequality (\ref{eq:stab2d}) since we have
$$
2^{-\frac{p-2}2}\int_\R u^2\; (|x|^p)'' \ee^{-\frac 2{\sqrt{2}^{p}}|x|^p}dx-\int_\R u'^2\; \ee^{-\frac 2{\sqrt{2}^{p}}|x|^p}dx=2^{\frac 32}2^{-\frac 1p}\,(p-2)\ \Gamma\Big(\frac 1p\Big)>0,
$$
where $\Gamma(\cdot)$ indicates the Gamma function. Hence for the measure $\nu_p^2$, with $p>2$, there is no non-coordinate half-space which is stable and the same happens for the $N$-fold tensorized measure $\nu_p^N$ by Theorem \ref{stabteononcoordinate}.
\end{proof}

%%%%%%%%%%%%%%%%%%%%%%%%%%%%%%%%%%%%%%%%%%%%%%%%%%%%%%%%%%%%%%%%%%%%%%%%%%%%%%%%%%%%%%%%%%%%%%%%%%%%%%%%%%%%%%%%%%%%%%%%%%%%%%%%%%%%%%%%
\subsection{The logistic measure}
%%%%%%%%%%%%%%%%%%%%%%%%%%%%%%%%%%%%%%%%%%%%%%%%%%%%%%%%%%%%%%%%%%%%%%%%%%%%%%%%%%%%%%%%%%%%%%%%%%%%%%%%%%%%%%%%%%%%%%%%%%%%%%%%%%%%%%%%
Let us denote by $\lo$ the \logi measure on $\R$, $d\lo(x)=f(x)\, dx=e^{-V(x)}dx$, where
$$
f(x)=\frac{\ee^x}{(1+\ee^x)^2}, \quad x\in \R.
$$
The \logi measure is a symmetric log-concave measure with %Gaussian behaviour around the origin, and
exponential tails, moreover $V''(x)=2f(x) >0$ for every $x\in\R$, with $\inf V''=0$.
We are going to prove that $\lo$ satisfies inequality (\ref{isoGeneric}), finding an explicit value for the constant $\C{\lo}$. Moreover we will show that the $N$-times product $\lo^N$ has stable non-coordinate half-spaces only in dimension 2.

% Notice that, thanks to Bobkov's result \cite{B1},  half-lines solve the \isopp for the \logi measure, and hence its distribution function $x(t)$  solves the differential equation $x'=x(1-x)$ (which is of the logistic type), which leads to
% $$
% \I{\lo}(t)=t(1-t) \qquad\text{ for every }t\in[0,1].
% $$

%%%%%%%%%%%%%%------------------------------------------------------------------------------------------------------------------------
\subsubsection{Spectral gap of the logistic measure}
%%%%%%%%%%%%%%------------------------------------------------------------------------------------------------------------------------

\begin{proposition}\label{calcolola}
 The best constant in the Poincar\'e inequality (\ref{poinc}) for  the \logi measure on $\R$ is
$$
\la_{\lo}=\frac 14\cdot
$$      
\end{proposition}

%%-----------
\begin{proof}
%%-----------
In order to compute $\la_{\lo}$, we will first show that $\la_{\lo}\ge \frac 14$, and then show that equality holds using an approximation method.

Since $\va(u+c)=\va(u)$ for all real constants $c$ and $\va(u)\le \|u\|_{L^2_{\lo}}$, we have
\begin{eqnarray}\label{lalo}    
\la_{\lo} &=& \inf_{u\in \wud{\lo}(\R)} \frac{\int_{\R} u'^2(x)d\lo(x)}{\va(u)}=
                \inf_{\substack{u\in \wud{\lo}(\R) \\ u(0)=0}} \frac{\int_{\R} u'^2(x)d\lo(x)}{\va(u)} \\
      &\ge& \inf_{\substack{u\in \wud{\lo}(\R) \\ u(0)=0}} \frac{\int_{\R} u'^2(x)d\lo(x)}{\int_{\R}u^2(x)d\lo(x)}.\nonumber
\end{eqnarray}
Notice that
$$
\inf_{\substack{u\in \wud{\lo}(\R)\\ u(0)=0}} \frac{\int_{\R} u'^2(x)d\lo(x)}{\int_{\R }u^2(x)d\lo(x)} =
     \inf_{\substack{u\in \wud{\lo}(\R)\\ u(0)=0}} \lim_{b\to \infty} \frac{\int_{-b}^b u'^2(x)d\lo(x)}{\int_{-b}^b u^2(x)d\lo(x)}.
$$
Hence to prove that $\la_{\lo}\ge\frac 14$ it is sufficient to show that
\begin{equation*}
\la_0^b =  \inf_{\substack{u\in \wud{\lo}([-b,b])\\ u(0)=0}} \frac{\int_{-b}^b u'^2(x)d\lo(x)}{\int_{-b}^b u^2(x)d\lo(x)}\ge \frac 14,
\end{equation*}
for every $b\in \R$.
Let us denote by $J(v)$ the functional defined on  $\wwb=\{ u\in \wud{\lo}([-b,b]),\ u(0)=0, u\not\equiv 0 \}$:
$$
J(u) = \frac{\int_{-b}^b u'^2(x)d\lo(x)}{\int_{-b}^b u^2(x)d\lo(x)};
$$
since $J(u)=J(|u|)$, we may assume $u\ge 0$.
For every $b\in\R$ the density $f$ of $\lo$ is bounded from above and below in $[-b,b]$ by two positive constants, so that we can apply standard arguments of compactness and lower semi-continuity in Sobolev spaces, and obtain the existence of a minimum for $J$ in $\wwb$.
Let $u$ be a minimizing function, i.e. $u\in\wwb$, and $J(u)=\la_0^b$.
We can now deduce the Euler equation for the above minimum problem and we obtain that $u$ solves the following
\begin{equation}\label{pb1}
\begin{cases}
u''(x)-V'(x)u'(x) &=-\la_0^b u(x), \qquad x\in [-b,b]\\ 
     u(0)=0 &\\
     u'(\pm b)=0, &
\end{cases}
\end{equation}
where $V'(x)= \tanh(x/2)$.
Since $V'$ is odd, without loss of generality we may assume that $u$ is even and hence we reduce the study of (\ref{pb1}) to the interval $[0,b]$ (indeed, for each function $u$ solving (\ref{pb1}), $u(x)+u(-x)$ is again a solution).
Moreover, problem (\ref{pb1}) can be seen as the Sturm-Liouville eigenvalues problem
\begin{equation}\label{pb2}
\begin{cases}
w''(x) +\Big( \frac 1{2\cosh^2(\frac x2)}+\la -\frac 14 \Big)w(x)=0, \qquad x\in [0,b]\\        
     w(0)=0 & \\
     w'(b) + \frac 12\tanh(\frac b2) w(b)= 0, &
\end{cases}
\end{equation}
corresponding to the eigenvalue $\la=\la_0^b$, via the transformation $w(x)= \frac{u(x)}{\cosh(\frac x2)}$.
It is known that there exist countably many values of $\la$ such that (\ref{pb2}) has a (non trivial) solution.
In particular $\la_0^b$ is the smallest positive one, and hence we may assume $w$ to be positive in $(0,b)$ (see for example \cite{H}) which implies the positivity of $u$ in $(0,b)$ too.

Notice that $\sinh(\frac x2)$ solves the equation in (\ref{pb1}) with $\la_0^b=\frac 14$, but it does not belong to $\wud{\lo}(\R)$, and hence it can not be a solution of the problem in the whole $\R$.
However, using an heuristic argument, it can give an idea of the behaviour of the possible eigenvalues.
In fact, roughly speaking, an eigenfunction has as much oscillations, as bigger is its corresponding eigenvalue (see again \cite{H}); hence, since $\sinh(x/2)$ does not oscillate at all in $(0,+\infty)$, its corresponding constant $\frac 14$ can be seen as a lower bound for the possible eigenvalues.
Hence we can expect $\la_0^b\ge\frac 14$.

More precisely, let us consider the Wronskian determinant of $u$ and $\sinh(x/2)$:
$$
\textbf{W}(x)=u'(x)\sinh(x/2)-\frac 12 u(x)\cosh(x/2).
$$
It solves the following differential equation,
$$
\textbf{W}'(x)=\textbf{W}(x)\,\tanh(x/2)+\Big(\frac 14-\la_0^b\Big) u(x)\sinh(x/2),
$$
with $\textbf{W}(0)=0$.
Hence
$$
\textbf{W}(x)=(\frac 14-\lab^b)\cosh(x/2)\left(\int_0^xu(y)\tanh(y/2)\,dy\right),
$$
and its sign depends only on the sign of $(\frac 14-\la_0^b)$.
As $\textbf{W}(b)<0$, we must have $\frac 14 -\la_0^b\le 0$.
Hence $\la_0^b\ge \frac 14$, which implies $\la_{\lo}\ge \frac 14$.

In order to complete the proof we show that $\frac 14$ is in fact the infimum of (\ref{lalo}).
Indeed, consider the sequence
$$
u_{\ep}(x)=\sinh\left(\frac x2(1-\ep)\right);
$$
notice that $u_{\ep}\in\wud{\lo}(\R)$ for every $\ep>0$.
Moreover, $J(u_{\ep})$ (for $b=\infty$) converges to $\frac 14$, as $\ep$ tends to zero, as
$$
\|u'_{\ep}\|^2_{\rm{L}^2_{\lo}(\R)}=\frac{(1-\ep)^2}{4}\left(1+\|u_{\ep}\|^2_{\rm{L}^2_{\lo}(\R)} \right).
$$
\end{proof}
%%%------------
%%%%%%%%%%%%%%%%%
\begin{remark}
%%%%%%%%%%%%%%%%%
Notice that equality in the Poincar\'e inequality can not hold, that is, $\frac 14$ is not a minimum of $J$ on $\R$.
Indeed, if there exists a function ${v}$ for which equality holds in (\ref{poinc}) with $\tau=\lo$, $\la=\frac 14$, then $v$ would be a solution of the differential equation in (\ref{pb1}) on $\R$ with $\la_0^b=\frac 14$) and $v(0)=0$. But this implies that $v$ is a non zero multiple of $\sinh(x/2)$ in $\R$ and hence $v\not\in \rm{L}^2_{\lo}(\R)$.
\end{remark}

%%%%%%%%%%%%%%%%%%%%%%%%%-----------------------------------------------------------------------------------------------------------------------------
\subsubsection{Dimension-free isoperimetric inequalities for $\lo$}
%%%%%%%%%%%%%%%%%%%%%%%%%-----------------------------------------------------------------------------------------------------------------------------
Since $\lo$  has a log-concave density, a  result of Bobkov \cite{B1} ensures that half-lines solve the \isopp for the \logi measure.
Moreover its distribution function $F(x):=F_\mu(x)=e^x/(1+e^x)$ verifies a differential equation of the logistic type: $F'=F(1-F)$.
It follows that for all $t\in[0,1]$,
$$
\I{\lo}(t)=t(1-t).
$$

The starting point to prove a dimension-free \isopi  for the \logi measure  is the following result,
which originates from the paper \cite{L} but incorporates a numerical improvement given in \cite{RW}
 (see also the proof of Theorem 19 in \cite{BCR1}).
%%%%%%%%%%%%%%%%%%%%%%%%%%%%%
\begin{theorem}\label{bcr}
%%%%%%%%%%%%%%%%%%%%%%%%%%%%%
Let $\tau$ be a probability measure on $\rn$, with $d\tau(x)=\ee^{-V(x)}\,dx$, $V\in C^2(\rn)$ and $D^2V\ge 0$ on $\rn$.
For any Borel set $A\subseteq \rn$ one has
\begin{equation}\label{BCR}
\bm{\tau}(\bd A)\ge \sqrt{\la_{\tau}}\ \cc\ \tau(A)\left( 1-\tau(A) \right),    
\end{equation}
where $\cc=\sup_{u\ge 0} \frac{1-\ee^{-2u}}{2\sqrt{u}}> 0,45125$.
\end{theorem}
Notice that (\ref{BCR}) can also be formulated as
$$
\I{\tau}(t)\ge \sqrt{\la_{\tau}}\ \cc\ t\left( 1-t \right) = \sqrt{\la_{\tau}}\ \cc\ \I{\lo}(t);
$$
this may be readily applied to the powers of the logistic distribution.
%%%%%%%%%%%%%%%%%%%%
\begin{theorem}\label{teoisoplogi}
%%%%%%%%%%%%%%%%%%%
Let $\lo$ be the \logi measure on $\R$. For every $N\ge 1$ and every $t\in[0,1]$ it holds
\begin{equation}\label{isoplo}
\I{\lon}(t) \ge \frac {\cc}2\ \I{\lo}(t),
\end{equation}  
where $\cc$ is the constant in (\ref{BCR}).
\end{theorem}

\begin{proof}
A direct application of Theorem \ref{bcr} yields
\begin{equation*}%\label{isoplo2}
\I{\lon}\ge \sqrt{\la_{\lon}}\ \cc\ \I{\lo}.
\end{equation*}
The tensorization property of the \poinc inequality, $\la_{\lo^N}=\la_{\lo}$, and Proposition \ref{calcolola}  guarantee that $\la_{\lo}=\frac 14$, which gives the claim.
\end{proof}

Recall the notation $
\I{\mu^{\infty}}(t) = \inf_N\I{\mu^N}(t).
$ This so-called infinite dimensional isoperimetric function is not easily estimated.
We have proved so far that
$$\I\lo\ge \I{\lo^{\infty}} \ge \frac{\cc}{2} \;\I\lo.$$
A similar estimate is given in  \cite{BH} for the symmetric exponential distribution instead of $\lo$, and with the constant $\frac 1{2\sqrt{6}}$ instead of $\frac {\cc}2$. Observe that
 $\frac {\cc}2> \frac1{2\sqrt{6}}$.
The upper-bound on the infinite dimensional isoperimetric profile can be estimated thanks to the following
well-known observation.
%%%%%%%%%%%%%%%%%%%%%%%%
\begin{proposition}\label{muinfinity}
%%%%%%%%%%%%%%%%%%%%%%%
Let $\mu$ be a probability measure on $\R$ having a density (with respect to the Lebesgue measure) and finite third moment. Let $\gamma$ be the standard Gaussian measure on $\R$. Let $\sigma^2= \var{\mu}(x)$. Then for every $t\in[0,1]$ it holds
\begin{equation}\label{asympt}
\I{\mu^{\infty}}(t) \le \frac 1{\sigma}\ \I{\gamma}(t).
\end{equation}
\end{proposition}

%%%------------
\begin{proof}
%%%------------
Notice that, since the \isopf is invariant with respect to translations of the measure, without loss of generality, we may assume
$
\int x\, d\mu(x) =0.
$
For any random variable $X$ with law $\mu$ and density $f_X$, we denote by $F_X$ its distribution function (that is $F_X(t)=\mu(X\le t)$).
On the probability space $(\R^N,\mu^N)$ the coordinate functions  $X_1,...,X_N$ can be viewed as  independent random variables with common law $\mu$.
Define
$$
\zn= \sum_{i=1}^N \frac{X_i}{\sqrt{N}}.
%% \qquad\text{ and }\qquad \zzn= \frac 1{\sigma} \zn= \sum_{i=1}^N \frac{X_i}{\sigma\sqrt{N}}.
$$
%% Hence
%% \begin{equation*}
%% F_{\zn}(x) =F_{\zzn}(\frac x{\sigma}),\qquad\text{ and }\qquad f_{\zn}(x) = \frac 1{\sigma} f_{\zzn}(\frac x{\sigma}).
%% \end{equation*}
Let us indicate by $H_y$ the half plane
$$
H_y= \left\{ x\in\rn\ |\ \sum_{i=1}^N \frac{x_i}{\sqrt{N}} \le y \right\}.
$$
We have $\mu^N(H_y)= F_{\zn}(y)$ and hence $\bm{\mu}(\bd H_y)=f_{\zn}(y)$.

Fix $t\in[0,1]$; for every $N\ge 1$ consider $y_N$ such that $\mu^N(H_{y_N}) = t$.
Hence for every $N\ge 1$ it holds
$$
\I{\mu^N}(t) \le \bm{\mu^N}(\bd H_{y_N}) = f_{\zn}(y_N).
$$
By the Local Limit Theorem for densities (see \cite{P})and the Berry-Esseen inequality (see \cite{C}),
\begin{eqnarray*}
&&| f_{\zn}(y_N) -  \frac 1{\sigma\sqrt{2\pi}} \ee^{-\frac{y_N^2}{2 \sigma^2}}| \to 0,\\
&& |\frac {y_N}{\sigma} - \fii^{-1}(t)| \to 0,
\end{eqnarray*}
where $\fii$ is the distribution function of the standard Gaussian measure. Hence
$$
f_{\zn}(y_N) \to \frac 1{\sigma} \fii'(\fii^{-1}(t)).
$$
As a conclusion, recalling inequality (\ref{ison}) it holds:
$$
 \I{\mu^{\infty}}(y)=\inf_{N} \I{\mu^N}(y) \le \frac 1{\sigma} \fii'(\fii^{-1}(y))= \frac 1{\sigma}\ \I{\gamma}(y).
$$
\end{proof}

Eventually  we get, for all $t\in[0,1],$
$$
\min\left\{ \I{\lo}(t),\frac{\sqrt{3}}{\pi}\I{\gamma}(t) \right\} \ge \I{\lo^{\infty}}(t)\ge \frac {\cc}2\ \I{\lo}(t),
$$
where the upper bound is slightly better than $\I{\lo} \ge \I{\lo^{\infty}}$, which consists essentially in testing the isoperimetric inequality on  coordinate half-spaces.
Notice that inequality (\ref{asympt}) comes out of evaluating the boundary measure of specific half-spaces in large dimension.
Getting better upper estimates on $\I{\lo^{\infty}}$ requires better test sets for which one can compute the measure and estimate the boundary measure.
A natural candidate would be the stable half-planes (see Section \ref{secstablogi}) $H=\{(x_1,x_2)\in\R^2:\, x_1\le x_2\}$; unfortunately, since $\bm{\lo^2}(H)=\frac{\sqrt{2}}6>\frac{\sqrt{3}}{\pi}\I{\gamma}(1/2)$, it does not give a better result.
%%%
%%%%
% \begin{figure}[h!]
% \begin{center}
% \centering
% \includegraphics[height=25mm]{Disegno2.jpg}
% \caption{The region where $\I{\lo^{\infty}}$ lives}
% \end{center}
% \end{figure}
%
%%%%%
\begin{center}
\begin{figure}[h]
\begin{tikzpicture}[x=8cm,y=8cm]
%%%%%
% asse x
\draw (-0.05,0)--(1.05,0);\draw (0,0) node[below]{$0$}; \draw (1,0) node[below]{$1$};%\draw (0,-0.03)--(0,0.27);
\newcommand{\disegnologi}{(0,0) to[out=60,in=210] (0.15, 0.13) to[out=30,in=180] (0.5, 0.21) to[out=0,in=150] (0.85, 0.13) to[out=330,in=120] (1, 0)}
\newcommand{\illogi}{[domain=0:1] plot({\x}, {(\x)*(1-\x)})}
\begin{scope}
\clip \illogi;
\fill[gray!30] \disegnologi;
\end{scope}
\fill[white] [domain=0:1] plot({\x}, {9*\x*(1-\x)/40});
\draw[blue, dashed] \disegnologi;
\draw[blue, ->, dashed] (1.05,0.12) node[right]{\small$\frac{\sqrt{3}}{\pi}\I{\gamma}(t)$}--(0.9,0.12);
\draw[blue!40!black!100!, thick] [domain=0:1] plot({\x}, {(\x)*(1-\x)});
\draw[blue!40!black!100!, ->] (1.05,0.24) node[right]{\small$\I{\lo}(t)$}--(0.65,0.24);
\draw[blue!60!black!100!, thick] [domain=0:1] plot({\x}, {9*\x*(1-\x)/40});
\draw[->,blue!60!black!100!] (1.05,0.03) node[right]{\small$\frac {\cc}2\ \I{\lo}(t)$} --(0.9,0.03);
\end{tikzpicture}
\caption{The region where $\I{\lo^{\infty}}$ lives}\label{quasiregFigGen}
\end{figure}
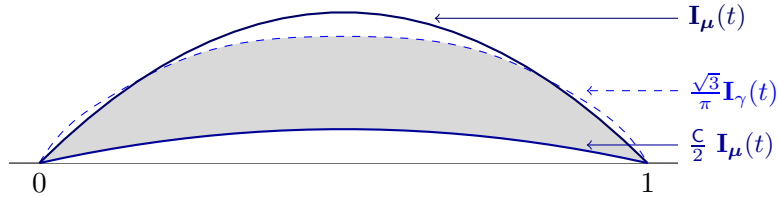
\end{center}

%%%%%%%%%%%%%%%%%%%%%%%%%%%%%%%%%%%%%%%%%%%%%%%%%%%%%%%%%%%%%%%%%%%%%%%%%%%%%%%%%%%%%%%%%%%%%%%%%%%%%%%%%%%%%%%%%%%%%%%%%%%%%%%%%%%%%%%%%%
\subsubsection{Stationarity and stability of half-spaces for the logistic measure}\label{secstablogi}
%%%%%%%%%%%%%%%%%%%%%%%%%%%%%%%%%%%%%%%%%%%%%%%%%%%%%%%%%%%%%%%%%%%%%%%%%%%%%%%%%%%%%%%%%%%%%%%%%%%%%%%%%%%%%%%%%%%%%%%%%%%%%%%%%%%%%%%%%%

As corollaries of Theorem \ref{stazteo} and Theorem \ref{stabteocoordinate}, Theorem \ref{stabteononcoordinate} we can give a description of half-spaces which are stationary and stable for the \logi measure, respectively.
Consider the half-space $H^{N+1}_{{\vv},t}=\{x\in\R^{N+1}\ :\ \langle x; \vv\rangle< t\}$; the following holds.
%%%%%%%%%%%%%%%%%%%%%%%
\begin{proposition}
%%%%%%%%%%%%%%%%%%%%%%%
The half-space $H^{N+1}_{{\vv},t}$ is stationary for the \logi measure if and only if it is either a coordinate half-spaces or $\vv\in\vn{N}$ and $t=0$.
\end{proposition}
The proof immediately follows by Theorem \ref{stazteo}.

%%%%%%%%%%%%%%%%
\begin{theorem}
%%%%%%%%%%%%%%%%%
A coordinate half-space $H^{N+1}_{t}$ is stable if $|t| \ge 2\log(2+\sqrt{3})$.
A non-coordinate half-space $H^{N+1}_{\vv,t}$ is stable   if and only if $N+1=2$, $t=0$ and  $\vv\in\vn{N}$.
\end{theorem}
%%%------------
\begin{proof}
%%%------------
The coordinate case immediately follows by Theorem \ref{stabteocoordinate} and Proposition~\ref{calcolola}.

Let us consider the non coordinate case; as stationarity is a necessary condition for stability, the only choice is $H^{N+1}_{\vv,0}$ with $\vv\in\vn{N}$.
In particular, by the symmetry of the measure and Theorem \ref{stabteononcoordinate}, it is enough to show that $H^{N+1}_{\tilde{\vv},0}$ satisfies (\ref{stab1}) and not (\ref{stab2}), so that it is stable for $N+1=2$ but it is not for $N+1\ge 3$.

\noindent{\it The planar case $N+1=2$}.
By definition (and some changes of variables) the half-space $H=H^2_{{\vv},0}$ is stable for $\lo^2$ if and only if
%%%
\begin{equation}\label{Qge0Logi}
\int_{\R}u^2(x)(-\psi''(x))\ee^{2\psi(x)}\,dx\le \frac 12 \int_{\R}u'^2(x)\ee^{2\psi(x)}\,dx ,
\end{equation}
for each $u\in C^{\infty}_0(\R)$ such that
\begin{equation}\label{zeromean}
\int_{\R}u(x)\ee^{2\psi(x)}\,dx=0.
\end{equation}
%%%
Notice that, since $\psi''(x)=-2f(x)$, we have $D^2\phi(x_1,x_2)= -2 \delta_{ij}f(x_i)$, and condition (\ref{Qge0Logi}) becomes
$$
\int_{\R}4 u^2(x)f^3(x)\,dx \le \int_{\R}u'^2(x)f^2(x)\,dx.
$$
Again we reduce the problem to compact intervals $\qn{1}{k}=[-k,k]$.
Notice that, thanks to Proposition \ref{propodd}, we are allowed to replace the zero mean condition (\ref{zeromean}) with the assumption that $u$ is odd.
Hence we define
\begin{equation}\label{lak2}
\lab^2 = \inf\left\{ \int_{\qn{1}{k}} u'^2f^2\ :\  u\in C^{\infty}_0(\qn{1}{k}), \int_{\qn{1}{k}} u^2f^3=1, \ u\text{ is odd}\right\},
\end{equation}
and the stability condition for $H^N_{\tilde{\vv},0}$ follows if we prove
\begin{equation}\label{aim}
\lab^2\ge 4.    
\end{equation}
Notice (\ref{Qge0Logi}) is of the same type as the Brascamp and Lieb inequality (\ref{BrLieb}), with the weight function on the opposite side. The idea is then to find an appropriate probability measure $\ee^{-g}$ such that the weight can be, in some sense, reversed. Consider
$$
d\tau(x)=30 f^3(x)dx=  30 \ee^{-g(x)}dx,
$$
where $g(x)=3\ln(\frac{(1+\ee^x)^2}{\ee^x})$. Hence  $g''(x)=3\frac{2\ee^x}{(1+\ee^x)^2}= 6 f(x)$ and
then(\ref{BrLieb}) gives that for every $u\in C^1$, $u$ odd and $\var{\tau}(u)<\infty$ it holds:
$$
\int_{\R} u^2(x) \,f^3(x)\,dx \le \frac 16 \int_{\R} u'^2(x)\frac 1{f(x)}\, f^3(x)\,dx.
$$
%% $$
%% \int_{\R} u^2(x) \,f^3(x)\,dx -30\left(  u(x) f^3(x)\, dx \right)^2 \le \frac 16 \int_{\R} u'^2(x)\frac 1{f(x)}\, f^3(x)\,dx.
%% $$
This implies that for every compact interval $\qn{1}{k}\subseteq \R$ and for every odd function $u\in C^{\infty}_0(\qn{1}{k})$,
\begin{equation*}
\frac{\int_{\qn{1}{k}} u'^2(x)f^2(x)\,dx}{\int_{\qn{1}{k}} u^2(x)f^3(x)\,dx}\ge 6,
\end{equation*}
which entails $\lab^2\ge 6$, for every $k\in\R$ and hence (\ref{aim}) holds.

\noindent{\it The $N+1$ dimensional case, $N+1\ge 3$.}
We are going to show that (\ref{stab2}) does not hold.
Indeed, recall that $d\lo(x)=f(x)dx=\ee^{-V(x)}dx$, with $V(x)>1$; using Remark \ref{la<v2} we get
\begin{equation}\label{stablog3}
\la_\lo\le \dfrac{\int_\R V''(x)d\lo(x)}{\int_\R d\lo(x)}.
\end{equation}
Assume (\ref{stab2}) holds true; it follows
$$
\la_\lo\ge \dfrac{\int_\R V''(x/\sqrt{2})f^2(x/\sqrt{2}) \, dx}{\int_\R f^2(x/\sqrt{2})\, dx}= \dfrac{\int_\R V''(x)f^2(x)\,dx}{\int_\R f^2(x)\,dx}.
$$
Hence by (\ref{stablog3}) we obtain
$$
\dfrac{\int_\R V''(x)f^2(x) \, dx}{\int_\R f^2(x)\, dx}\le \dfrac{\int_\R V''(x)f(x)\,dx}{\int_\R f(x)\,dx},
$$
which gives, recalling that for the logistic $V''(x)=2f(x)$,
$$
\int_\R f(x)\,dx\ \int_\R f^3(x)\, dx \le \left(\int_\R f^2(x)\,dx\right)^2.
$$
Notice that this latter does not hold true by H\"older inequality (consider the functions $f^{3/2}, f^{1/2}$; as they are not proportional the strict sign in H\"older's inequality holds), hence $H^M_{{\vv},0}$ is not stable for any $M\ge 3$.
\end{proof}
%%%------------
%%%%%%%%%%%%%%%%%%%%%%%%%%%%%%%%%%%%%%%%%%%%%%%%%%%%%%%%
The next result asserts that in the stability condition (\ref{Qge0Logi}) we can replace the zero mean assumption (\ref{zeromean}) considering the class of odd functions.

\begin{proposition}\label{propodd}
There exists a solution to the Euler-Lagrange equation for the minimum problem
\begin{equation*}
\inf\left\{ \int_{\qn{1}{k}} u'^2f^2\ :\  u\in C^{\infty}_0(\qn{1}{k}), \int_{\qn{1}{k}} u^2f^3=1, \ \int_{\qn{1}{k}}u f^2 =0\right\},
\end{equation*}
which is odd.   
\end{proposition}
%%%%%%%%%%%%%%%%%%%%%%%%%%%%%%%%%%%%%%%%%%%%%%%%%%%%%%%%
%%%%----------------
\begin{proof}
%%%%----------------
Consider the Euler-Lagrange equation for (\ref{lak2}):
\begin{equation}\label{ppb}
\begin{cases}
& (u'f^2)'+\la u(-\psi'')f^2=0\\
& u'(\pm k)=0,  
\end{cases}
\end{equation}
and perform the following change of variable
$$
dt=\frac{dy}{h(y)},\qquad\text{that is }\qquad t=\int_0^y\frac 1{h(s)}\,ds,
$$      
where $y=y(t)$.
Hence,
$$
\dfrac{d}{dy} = \dfrac 1{h(y(t))}\dfrac{d}{dt}\ ,
$$
and Problem (\ref{ppb}) becomes
\begin{equation}\label{ppb2}
\begin{cases}
& \dfrac 1{h}\Big( u_t\dfrac{f^2}{h} \Big)_t +\la u (-\psi'')f^2=0\\
& u_t(\pm\tilde{k})=0,
\end{cases}
\end{equation}
where $\tilde{k}=\int_{0}^k\frac 1{h(s)}\,ds$.
Choosing $h(y)= \frac 1{f^2(-\psi'')}$, we get
$$
\Big( u_t\frac{f^2}{h}\Big)_t +\la u =0,
$$
which is, together with the boundary conditions, a Sturm-Liouville problem of the type considered in \cite{H}.

Notice that, as $h(y)>0$, problem (\ref{ppb2}) is equivalent to problem (\ref{ppb}).
By Theorem 4.1 in \cite{H}, we have that the first eigenfunction $u\big(y(t)\big)$, corresponding to the first eigenvalue $\la=\lab^2$, has a unique zero in $(-\tilde{k},\tilde{k})$.
This implies that $u(y)$ has a unique zero in $(-k,k)$.

Consider again problem (\ref{ppb}); notice that $w(y)=u(y)-u(-y)$ solve it.
Moreover, thanks to condition (\ref{zeromean}),  $u$ has zero mean in $(-k,k)$ (w.r.t. the measure $f^2(y)\,dy$), and, as noticed above, it has a unique zero in $(-k,k)$.
This implies that $u$ can not be even, and hence $w$ is not constant null in $(-k,k)$.
We can then assume the solution $u$ to be odd.
\end{proof}

%%%%%%%%%%%%%%%%%%%%%%%%%%%%%%%%%%%%%%%%%%%%%%%%%%%%%%%%%%%%%%%%%%%%%%%%%%%%%%%%%%%%%%%%%%%%%%%%%%%%%%%%%%%%%%%%%%%%%%%%%%%%%%%%%%%%%%%%
\subsection{Perturbations of Gaussian measures}
%%%%%%%%%%%%%%%%%%%%%%%%%%%%%%%%%%%%%%%%%%%%%%%%%%%%%%%%%%%%%%%%%%%%%%%%%%%%%%%%%%%%%%%%%%%%%%%%%%%%%%%%%%%%%%%%%%%%%%%%%%%%%%%%%%%%%%%%

In this part we present heuristic evidence for the existence of non-Gaussian probability distributions for which non-coordinate hyperplanes are stable for their product measures in any dimension. The tools of perturbation theory
would certainly allow to turn the following sketch into a rigorous argument.

Since for the standard Gaussian measure all hyperplanes are stable, it is natural to look for perturbations  of it.
We look for a measure $d\tau=\ee^{-v(x)}dx$ on $\R$ such that conditions \emph{(\ref{p1})} and \emph{(\ref{p2})} in Theorem \ref{teoPoincpeso} hold for $d\tau$, $d\nu(x)=\ee^{-2v(x/\sqrt{2})}$, $\Theta(x)=v''(x/\sqrt{2})$.
According to the results of the previous sections, this is equivalent to construct a measure $\tau$ such that the half-space through the origin, orthogonal to $(0,...,0,\pm 1/\sqrt{2},\pm 1/\sqrt{2})$ is (stationary and) stable for $\tau^{N+1}$ for every $(N+1)\ge 3$. As we said, we choose $\tau=\te$ as a perturbation of the Gaussian measure, i.e.
$$
d\te=\ee^{-(\frac{x^2}{2}+\ep\psi(x))}dx=\ee^{-v_\ep(x)}dx,
$$
where $\psi\in C^{\infty}_0(\R)$ will be suitably determined later on. In order to preserve symmetry, we choose $\psi$ even. As $\psi$ has compact support, for sufficiently large $x$, $\te$ coincides with the (not normalized) Gaussian measure for every $\ep>0$ and hence it satisfies an \isopi of the type (\ref{isoGeneric}).
Moreover, up to normalization, $\tau_0$ coincides with the Gaussian measure. For $\ep>0$ let us define:
%%%
\begin{eqnarray}
&& \la(\ep)=\inf\left\{  \frac{\int_{\R}u'^2(x)d\te(x)}{\int_{\R}u^2(x)d\te(x)}\ :\ \int_{\R}ud\te=0,\ u\in \wud{\te}(\R) \right\}, \label{laep}\\
&& k(\ep)=\inf\left\{  \frac{\int_{\R}u'^2(x)d\nne(x)}{\int_{\R}u^2(x) v''_\ep({\textstyle \frac x{\sqrt{2}}})d\nne(x)}\ :\ \int_{\R}ud\nne=0,\ u\in \wud{\te}(\R) \right\}, \label{kep}\\
&& a(\ep)=\sup\left\{ \frac{ \int_{\R} u^2(y)v''_\ep({ \textstyle \frac y{\sqrt{2}} })d\nne(y) - \int_{\R} u'^2(y)d\nne(y)}{\int_\R u^2(y)d\nne(y)}\ :\ u\in \wud{\te}(\R), u\not\equiv 0  \right\},\label{aep}
\end{eqnarray}
where $v_\ep(x)=\frac{x^2}{2}+\ep\psi(x)$, $d\nne(x)=\ee^{-2v_\ep(x/\sqrt{2})}$.
Here we assume, in particular, that $\ep>0$ is sufficiently small so that $v''_\ep>0$ in $\R$ (recall that $\psi$ has compact support).
Moreover we assume the infima and supremum in (\ref{laep})-(\ref{aep}) to be achieved and that the following derivatives exist:
$$
\dot{\la}=\frac{d}{d\ep} \la(\ep)|_{\ep=0},\qquad \dot{k}=\frac{d}{d\ep} k(\ep)|_{\ep=0},\qquad \dot{a}=\frac{d}{d\ep} a(\ep)|_{\ep=0}.
$$
Note that $\la(0)$ is the spectral gap of the Gaussian measure, hence $\la(0)=1$.
Analogously $k(0)=1$. For $a(0)$ we have
$$
a(0)=\sup\left\{ 1-\frac{\int_\R u'^2 d\nu_0}{\int_\R u^2 d\nu_0}\ :\ u\in \wud{\te}(\R), u\not\equiv 0  \right\}=1.
$$
As a first consequence, if $\dot{k}>0$ then condition (\ref{p1}) in Theorem \ref{teoPoincpeso} is verified for $\ep>0$ small enough. If moreover $\dot{\la}>\dot{a}$, then condition (\ref{p2}) also hold for sufficiently small $\ep>0$.
In what follow we compute $\dot{\la}$, $\dot{k}$ and $\dot{a}$ in order to find a function $\psi\in C^{\infty}_0(\R)$ such that
$$
\dot{k}>0,\qquad \dot{\la}>\dot{a}.
$$
We start with the analysis of $\dot{\la}$; the corresponding calculations for $\dot{k}$, $\dot{a}$ are analogous and we omit them.

The Euler-Lagrange equation of the minimum problem (\ref{laep}) is
$$
u''_\ep-(x+\ep\psi')u'_\ep+\la(\ep)u_\ep = 0,
$$
verified by a minimizer $u_\ep$.
We differentiate it with respect to $\ep$ the equation at $\ep=0$, writing $u$ for $u_0$, $\la$ for $\la_0=1$ and $\dot{u}$ for the derivative of $u$ with respect to $\ep$ at $\ep=0$. We get
\begin{equation}\label{upunto}
\dot{u}''-x \dot{u}'-\psi'u'+\dot{\la}u+\la\dot{u}=L\dot{u}-\psi'u'+\dot{\la}u+\la\dot{u}=0,
\end{equation}
where the operator $L$ is defined by $Lw=w''-xw'$.
Note that $L$ is self-adjoint with respect to $\tau=\tau_0$:
$$
\int_\R w Lvd\tau=\int_\R vLwd\tau, \qquad\text{ for every }w,v.
$$
Integrating equation (\ref{upunto}) against $ud\tau$ we get
$$
\int_\R uL\dot{u}d\tau -\int_\R\psi'uu'd\tau+\dot{\la}\int_\R u^2d\tau+\la\int_\R u\dot{u}d\tau=0.
$$
On the other hand
$$
\int_\R uL\dot{u}d\tau +\la \int_\R u\dot{u}d\tau = \int_\R\dot{u}(Lu+\la u)d\tau=0,
$$
as $Lu+\la u =0$.
Hence
$$
\dot{\la} \int_\R u^2 d\tau = \int_\R \psi' uu'd\tau.
$$

We recall that $u=u_0=x$ and we introduce the normalized Gaussian measure $d\gamma(x)=\frac 1{\sqrt{2\pi}} \ee^{-x^2/2}\,dx$.
We get
\begin{eqnarray*}
\dot{\la}&=&\int_\R \psi'(x)\;u(x)\,u'(x),d\gamma(x)=\int_\R \psi'\; \Big(\frac{u^2}2\Big)' d\gamma=-\int_\R \psi\; L(u^2/2)\,d\gamma\\
         &=&\int_\R \psi(x) (x^2-1) d\gamma(x)=\frac 1{\sqrt{2\pi}}\int_\R \psi(x)(x^2-1)\ee^{-\frac{x^2}2}dx.
\end{eqnarray*}
In a similar way we can compute $\dot{k}$, $\dot{a}$, obtaining
\begin{eqnarray*}
\dot{k}&=& 2\int_\R \psi(x/\sqrt{2}) (-x^4+6x^2-3)d\gamma(x)=\frac{2}{\sqrt{\pi}}\int_\R \psi(x)(-4x^4+12x^2-3)\ee^{-x^2}dx ,\\
\dot{a}&=& 2\int_\R \psi(x/\sqrt{2})(x^2-1)d\gamma(x)=\frac 2{\sqrt{\pi}}\int_\R \psi(x)(2x^2-1)\ee^{-x^2}dx.
\end{eqnarray*}

Hence the half-space $H^{N+1}_{\vv,0}$ is stable for the measure $\te$, for $\ep$ sufficiently small, if and only if there exists a function $\psi\in C^{\infty}_0(\R)$ such that
\begin{eqnarray*}
 \int_\R \psi(x)(-4x^4+12x^2-3)\ee^{-x^2}dx &>& 0,\\
  \frac 1{\sqrt{2}}\int_\R \psi(x)(x^2-1)\ee^{-\frac{x^2}2}dx &>& 2\int_\R \psi(x)(2x^2-1)\ee^{-x^2}dx.
\end{eqnarray*}
As $(-4x^4+12x^2-3)\ee^{-x^2}$ and $(x^2-1)\ee^{-\frac{x^2}2}-2\sqrt{2}(2x^2-1)\ee^{-x^2}$ are linearly independent functions in $L^2(\R)$, there exists $\psi\in C^{\infty}_0(\R)$ such that previous inequalities hold, and hence $H^{N+1}_{\vv,0}$ is stable for $\te$, for any $N+1\ge 2$.

%%%%%%%%%%%%%%%%%%%%%%%%%%%%%%%%%%%%%%%%%%%%%%%%%%%%%%%%%%%%%%%%%%%%%%%%%%%%%%%%%%%%%%%%%%%%%%%%%%%%%%%%%%%%%%%%%%%%%%%%%%%%%%%%%%%%%%%%%%%%%
\section*{Acknowledgements}
%%%%%%%%%%%%%%%%%%%%%%%%%%%%%%%%%%%%%%%%%%%%%%%%%%%%%%%%%%%%%%%%%%%%%%%%%%%%%%%%%%%%%%%%%%%%%%%%%%%%%%%%%%%%%%%%%%%%%%%%%%%%%%%%%%%%%%%%
This work began while the second author was supported by the RTN European Network ``Phenomena in High Dimensions'' and hosted by the Laboratoire de Statistique et Probabilit\'es at the Institut de Math\'ematiques, Universit\'e Paul-Sabatier of Toulouse (France).
She would like to thank Pawel Wolff for many helpful discussions and for his encouragement during her sojourn.
%%%%%%%%%%%%%%%%%%%%%%%%%%%%%%%%%%%%%%%%%%%%%%%%%%%%%%%%%%%%%%%%%%%%%%%%%%%%%%%%%%%%%%%%%%%%%%%%%%%%%%%%%%%%%%%%%%%%%%%%%%%%%%%%%%%%%%%%%%%%%
%%%%%%%%%%%%%%%%%%%%%%%%%%%%%%%%%%%%%%%%%%%%%%%%%%%%%%%%%%%%%%%%%%%%%%%%%%%%%%%%%%%%%%%%%%%%%%%%%%%%%%%%%%%%%%%%%%%%%%%%%%%%%%%%%%%%%%%%
%%%%%%%%%%%%%%%%%%%%%%%%%%%%%%%%%%%%%%%%%%%%%%%%%%%%%%%%%%%%%%%%%%%%%%%%%%%%%%%%%%%%%%%%%%%%%%%%%%%%%%%%%%%%%%%%%%%%%%%%%%%%%%%%%%%%%%%%
 
%%%%%%%%%%%%%%%%%%%%%%%%%%%%%%%%%%%%%%%%%%%%%%%%%%%%%%%%%%%%%%%%%%%%%%%%%%%%%%%%%%%%%%%%%%%%%%%%%%%%%%%%%%%%%%%%%%%%%%%%%%%%%%%%%%%%%%%%
%%%%%%%%%%%%%%%%%%%%%%%%%%%%%%%%%%%%%%%%%%%%%%%%%%%%%%%%%%%%%%%%%%%%%%%%%%%%%%%%%%%%%%%%%%%%%%%%%%%%%%%%%%%%%%%%%%%%%%%%%%%%%%%%%%%%%%%%         
\end{document}